\documentclass[letterpaper,11pt]{article}

\usepackage{fullpage}
\usepackage{color}

\usepackage{xspace,xcolor,enumerate}
\usepackage{amsmath,amssymb}
\usepackage{amsthm}
\usepackage[toc,page]{appendix}
\usepackage{thmtools}
\usepackage{thm-restate}
\usepackage{color,graphicx}
\usepackage{boxedminipage}
\usepackage{makecell}
\usepackage{tikz}

\definecolor{darkred}{rgb}{0.5,0,0}
\definecolor{darkgreen}{rgb}{0,0.35,0}
\definecolor{darkblue}{rgb}{0,0,0.55}

\usepackage[%
  pdfstartview=FitH,
  pdfpagemode=UseNone,
  colorlinks,
  linkcolor=darkblue,
  filecolor=darkred,
  citecolor=darkgreen,
  urlcolor=darkred,
]{hyperref}

\usepackage{subcaption} % For captions in subfloats
\usepackage[capitalise,nameinlink]{cleveref}
\usepackage{amsmath}
\usepackage{amssymb}
\usepackage{amsthm}
\usepackage{ytableau}
\usepackage{mathtools}
\usepackage{bbm}

\usepackage{dsfont}

\usepackage[inline]{enumitem}
\setlist[enumerate]{
  label = {(\roman*)},
  ref = {(\roman*)}}

\newtheoremstyle{thmstyle}%style name
  {\medskipamount}%space before
  {\smallskipamount}%space after
  {\slshape}%font used
  {0pt}%indentation
  {\bfseries}%modifier theorem head
  {.}%punctuation between theorem head and body
  { }%space after punctuation
  {\thmname{#1}\thmnumber{ #2}{\normalfont\thmnote{ (#3)}}}%theorem specifier

\newtheoremstyle{plainstyle}%style name
  {\medskipamount}%space before
  {\smallskipamount}%space after
  {\rmfamily}%font used
  {0pt}%indentation
  {\bfseries}%modifier theorem head
  {.}%punctuation between theorem head and body
  { }%space after punctuation
  {\thmname{#1}\thmnumber{ #2}{\normalfont\thmnote{ (#3)}}}%theorem specifier

\theoremstyle{thmstyle}
\newtheorem{theorem}{Theorem}[section]

\newtheorem{conjecture}[theorem]{Conjecture}

\newtheorem{corollary}[theorem]{Corollary}

\newtheorem{proposition}[theorem]{Proposition}

\newtheorem{fact}[theorem]{Fact}

\newtheorem{lemma}[theorem]{Lemma}

\newtheorem{claim}[theorem]{Claim}

\theoremstyle{plainstyle}

\newtheorem{definition}[theorem]{Definition}

\newtheorem{remark}[theorem]{Remark}
\newtheorem*{remark*}{Remark}

\newtheorem{notation}[theorem]{Notation}

\newcommand{\NN}{\mathbb{N}}

\newcommand{\RR}{\mathbb{R}}

\newcommand{\cB}{\mathcal{B}}
\newcommand{\cC}{\mathcal{C}}

\newcommand{\cM}{\mathcal{M}}
\newcommand{\cT}{\mathcal{T}}
\newcommand{\cU}{\mathcal{U}}
\newcommand{\cX}{\mathcal{X}}

\newcommand{\floor}[1]{\left\lfloor #1\right\rfloor}
\newcommand{\ceil}[1]{\left\lceil #1\right\rceil}

\DeclareMathOperator{\Cay}{Cay}
\DeclareMathOperator{\id}{id}
\DeclareMathOperator{\htt}{ht}
\DeclareMathOperator{\sgn}{sgn}

\DeclareMathOperator{\OPT}{OPT}
\DeclareMathOperator{\tr}{tr}

\begin{document}

\title{Tighter Bounds on the Independence Number\\ of the Birkhoff Graph}

\author{%
  Leonardo Nagami Coregliano\thanks{University of Chicago, lenacore@uchicago.edu}
  \and Fernando Granha Jeronimo\thanks{University of Chicago, granha@uchicago.edu. Supported in part by NSF grant CCF-1816372.}%
}

\date{\today}

\maketitle

\begin{abstract}
    The Birkhoff graph $\mathcal{B}_n$ is the Cayley graph of the
    symmetric group $S_n$, where two permutations are adjacent if they
    differ by a single cycle. Our main result is a tighter upper bound
    on the independence number $\alpha(\mathcal{B}_n)$ of
    $\mathcal{B}_n$, namely, we show that $\alpha(\mathcal{B}_n) \le
    O(n!/1.97^n)$ improving on the previous known bound of
    $\alpha(\mathcal{B}_n) \le O(n!/\sqrt{2}^{\; n})$ by
    [Kane--Lovett--Rao, FOCS~2017]. Our approach combines a
    higher-order version of their representation theoretic techniques
    with linear programming. With an explicit construction, we also
    improve their lower bound on $\alpha(\mathcal{B}_n)$ by a factor
    of $n/2$. This construction is based on a proper coloring of
    $\mathcal{B}_n$, which also gives an upper bound on the chromatic
    number $\chi(\mathcal{B}_n)$ of $\mathcal{B}_n$. Via known
    connections, the upper bound on $\alpha(\mathcal{B}_n)$ implies
    alphabet size lower bounds for a family of maximally recoverable
    codes on grid-like topologies.
\end{abstract}

\thispagestyle{empty}

\newpage
\pagenumbering{roman}
\tableofcontents
\clearpage

\pagenumbering{arabic}
\setcounter{page}{1}

\section{Introduction}\label{sec:intro}

A celebrated theorem of Birkhoff~\cite{Birkhoff46} characterizes the set of doubly stochastic matrices as
forming a convex polytope whose extreme points are permutation matrices. More precisely, if $M$ is an
$n$-by-$n$ matrix over the non-negative reals whose rows and column each sum to 1 (i.e., doubly stochastic),
then $M$ can be expressed as a convex combination of permutation matrices (i.e., matrices with exactly one
entry $1$ in each row and column and all other entries $0$). It is well known (see e.g.~\cite[Section~2]{BA96}
and~\cite[Section~II.5]{barvinok2002course}) that the skeleton of this polytope, called the Birkhoff graph
$\cB_n$, is the Cayley graph whose vertex set is the symmetric group $S_n$ and two permutations $\sigma$ and
$\tau$ are adjacent if and only if they differ by a single cycle, that is, $\sigma\tau^{-1}$ is a cycle. For
more properties of the Birkhoff graph and polytope we refer the reader to~\cite{Pak00,CM09}.

Recently, connections between the Birkhoff graph and coding theory, more specifically, the theory of maximally
recoverable codes~\cite{GHJY14}, were pointed out by Kane, Lovett and~Rao~\cite{lovett17}, who showed that the
alphabet size of a family of maximally recoverable codes on a grid-like topology (more precisely $T_{n\times
  n}(1,1,1)$ of~\cite{GHKSWY17}) is at least the chromatic number $\chi(\cB_n)$ of the Birkhoff graph, which
in turn, by the trivial bound, is at least $n!/\alpha(\cB_n)$, where $\alpha(\cB_n)$ is the independence
number of the Birkhoff graph\footnote{In~\cite[Claim~I.5]{lovett17}, this bound is presented directly in terms
  of the independence number $\alpha(\cB_n)$, but their proof in fact works for the chromatic number.}. Thus,
upper bounds for the independence number $\alpha(\cB_n)$ translate to lower bounds on the size of the alphabet
needed for such codes.

A well-known spectral technique to bound the independence number of a graph is the Hoffman
bound~\cite{Hoffman69}, which uses only the largest and the smallest eigenvalues of the graph. However, Kane,
Lovett and Rao~\cite{lovett17} observed that such spectral technique cannot yield a bound better than
$\alpha(\cB_n) \le O((n-1)!)$. This was their motivation to use stronger techniques based on representation
theory, which can be seen as a generalization of spectral theory. Using these techniques, they obtained the
following upper bound on $\alpha(\cB_n)$.

\begin{theorem}[\protect{\cite[Theorem~I.8]{lovett17}}]
  For every $n\in\NN_+$, we have
  \begin{align*}
    \alpha(\cB_n) & \leq \frac{n!}{\sqrt{2}^{\;n-4}}.
  \end{align*}
\end{theorem}

On the other side, the best lower bound known is given via an explicit construction of an independent set and
yields the following result\footnote{Their construction yields an independent set whose size is exactly the
  product in~\eqref{eq:KLRconstruction}, but only the rightmost bound is stated in~\cite{lovett17}.}.

\begin{theorem}[\protect{\cite[Theorem~I.7]{lovett17}}]
  For every $n\in\NN_+$ that is a power of $2$, we have
  \begin{align}\label{eq:KLRconstruction}
    \alpha(\cB_n)
    & \geq
    \prod_{i=1}^{\log_2(n)-1} 2^i!
    \geq
    \frac{n!}{4^n}.
  \end{align}
\end{theorem}

In this paper, we prove tighter asymptotic bounds on the independence number $\alpha(\cB_n)$ of $\cB_n$. Our
main result is the following.

\begin{restatable}{theorem}{TheoMainIndUB}\label{theo:main_ub_alpha}
  We have
  \begin{align*}
    \alpha(\cB_n) & \leq O\left(\frac{n!}{1.97^n} \right).
  \end{align*}
\end{restatable}

Our method consists of a generalization of KLR's representation theoretic approach combined with the solution
of a particular linear programming problem. As discussed above, this also
improves the lower bound for the size of the alphabet of a maximally recoverable code in the $T_{n\times
  n}(1,1,1)$ topology of~\cite{GHKSWY17} from $\Omega(\sqrt{2}^{\; n})$ to $\Omega(1.97^n)$.

On the other side, we improve KLR's construction of an independent set by a factor of $n/2$ when $n$ is a
power of $2$, and extend KLR's result for every $n$, which yields the following result.

\begin{restatable}{theorem}{TheoMainIndepConst}\label{theo:indep_set_construct}
  For every $n\in\NN_+$, we have
  \begin{align*}
    \alpha(\cB_n)
    & \geq
    \prod_{i=1}^{\floor{\log_2(n)}}\floor{\frac{n}{2^i}}!
    \geq
    \frac{n!}{4^n}\cdot 2^{\Theta((\log(n))^2)}.
  \end{align*}

  If $n$ is a power of $2$, then we can improve the bound above to
  \begin{align*}
    \alpha(\cB_n)
    & \geq
    \frac{n}{2}\prod_{i=1}^{\log_2(n)-1} 2^i!
  \end{align*}
\end{restatable}

In fact, we can construct a proper coloring of the Birkhoff graph such that each of the color classes is an
independent set achieving the bounds above.

\begin{restatable}{theorem}{TheoMainChromNumb}\label{theo:chrom_ub}
  Let $n$ be a positive integer. Then there is an explicit proper coloring establishing
  \begin{align*}
    \chi(\cB_n)
    & \le
    \prod_{i=0}^{\lceil \log_2(n) \rceil} \binom{\lceil n/2^i \rceil}{\lceil n/2^{i+1} \rceil}
    \le
    \frac{4^{n}}{2^{\Theta((\log(n))^2)}}.
  \end{align*}
  If $n$ is a power of $2$, then there is an explicit proper coloring strengthening the bound above to
  \begin{align*}
    \chi(\cB_n) & \leq \frac{2}{n}\prod_{i=1}^{\log_2(n)}\binom{2^i}{2^{i-1}}.
  \end{align*}
\end{restatable}

We believe that the techniques we introduce here for the upper bound should be strong enough to prove that
$\alpha(\cB_n)\leq O(n!/c^n)$ for any fixed constant $c\in(1,2)$, but a full theoretic proof seems quite
technical and elusive for now.

\begin{conjecture}
  There exists a constant $K > 0$ such that
  \begin{align*}
    \alpha(\cB_n) & \leq K\cdot\frac{n!}{(2-o(1))^n}.
  \end{align*}
\end{conjecture}

The paper is organized as follows. In \cref{sec:strategy}, we present
an overview of the proof of our main result,
\cref{theo:main_ub_alpha}. In \cref{sec:prelim}, we establish some
notation, recall some facts about representation theory and some
results we need from~\cite{lovett17}. Our main result will be a
combination of theoretical proofs done in \cref{sec:theo_proofs} and
computation which we describe in \cref{sec:computational}. Finally, we
construct an explicit independent set and a proper coloring in
\cref{sec:construction} yielding \cref{theo:indep_set_construct,theo:chrom_ub}.

% LocalWords:  Birkhoff polytope Cayley Lovett Rao Gopalan et al Pak KLR

\section{Proof Strategy}\label{sec:strategy}

In this section, we give an overview of our upper bound proof for the independence number
$\alpha(\cB_n)$. This proof builds on the representation theoretic techniques of Kane, Lovett and
Rao~\cite{lovett17}. Roughly speaking, our approach can be seen as higher-order version of KLR. To establish
an upper bound on $\alpha(\cB_n)$ of $t(n)$ (e.g., $t(n)=n!/c_0^n$ for $c_0 \in (1,2)$), it is enough to show
that any $A \subseteq S_n$ of size larger than $t(n)$ contains an edge in $\cB_n$. KLR classify the set $A$ in
a pseudorandom versus structured dichotomy. If $A$ meets the criteria of being pseudorandom, KLR use
representation theory to count the number of edges within $A$ corresponding to cycles of length precisely $n$
and show that this number is positive. Otherwise, they show that $A$ must have some structured subset $A'
\subseteq A$ that can be embedded in a edge preserving way into a smaller symmetric group $S_{n'}$, with $n' <
n$, and the proof concludes by an inductive argument. A crucial difference in our approach is that when $A$ is
pseudorandom, we are going to count edges corresponding to cycles of length\footnote{In fact, we only need to
  consider cycles of length $n-2i$ since we take $A$ to have permutations of the same sign and we take $n$ to
  be odd.}  $n,n-1,\ldots,n-\ell_0$ for some constant $\ell_0$ rather than only counting $n$-cycles. This is
precisely the sense in which our approach is a higher-order version of KLR. By considering all these
additional cycle lengths, the representation theoretic analysis becomes substantially more involved and more
ingredients are used as we detail below.

To a set $A \subseteq S_n$ one can associate a class function $\varphi_A$ whose precise definition is not
important for this high-level discussion. We denote by $\Psi_\ell \coloneqq \Psi_\ell(A)$ the number of edges
within $A$ corresponding to cycles of length exactly $n-\ell$. A simple representation theoretic argument
(easily derivable from KLR) establishes that $\Psi_\ell$ is proportional to
\begin{align*}
  \sum_{\lambda \vdash n} p^{\ell}_\lambda \cdot \chi^{\lambda}(\varphi_A),
\end{align*}
where $\{p^{\ell}_\lambda\}_\lambda$ are (explicit) coefficients over the reals and $\chi^{\lambda}$ is the
irreducible character associated to the partition $\lambda \vdash n$. In case $A$ is pseudorandom, the characters must
satisfy a set of linear constraints of the form
\begin{equation}\label{eq:linear_constraints}
  \left\{\sum_{\lambda \vdash n} k_{\lambda,m} \cdot \chi^{\lambda}(\varphi_A) \le c_m \right\}_{m \in \cM},
\end{equation}
where $k_{\lambda,m}$ and $c_m$ are (explicit) real coefficients. Set $M \coloneqq \max_{\ell}
\Psi_{\ell}$. In particular, if $M > 0$, then $A$ is certainly not independent as some edge count
$\Psi_\ell > 0$. Instead of working with the precise character evaluations $\chi^{\lambda}(\varphi_A)$
(abiding to representation theoretic rules), we can relax $\chi^{\lambda}(\varphi_A)$ to be arbitrary real
variables\footnote{Actually, the structure of $\varphi_A$ forces $\chi^\lambda(\varphi_A) \ge 0$ and thus
  $x_\lambda$ can be taken to be non-negative.} $x_\lambda$ only bound to satisfy the linear
constraints~\eqref{eq:linear_constraints}. By considering the objective function $\min_{x_\lambda} M$, we
have a linear program on our hands. Similarly, if the optimum value of this linear program is positive, we are
again certain that $A$ is not independent since we are dealing with a relaxed minimization problem.

The linear program mentioned above is actually somewhat more delicate as the coefficients $k_{\lambda,m}$ and
$p^{\ell}_\lambda$ depend on the degree $n$ of $S_n$. Fortunately, for ``low complexity'' shapes $\lambda
\coloneqq (n-\sum_{i=2}^s \lambda_i,\ldots,\lambda_s)$, those in which $\sum_{i=2}^s \lambda_i$ is a constant,
the coefficients $k_{\lambda,m}$ and $p^\ell_\lambda$ become fixed constants provided $n$ is sufficiently
large. To be able to work in this asymptotic regime where these coefficients of some low complexity shapes are
fixed, we have to additionally require $n$ large when $A$ is pseudorandom. For this reason, we will need to
control the total loss in cardinality when $A$ is structured. Recall that the structured case is handled by
finding some $A' \subseteq A$ which is then embedded in $S_{n'}$ for some $n'$ possibly much smaller than
$n$. To ensure that $n'$ is still arbitrarily large we exploit a density increment phenomenon in the
structured case, namely, we observe that this $A'$ satisfies $\lvert A' \rvert/n'! > c^{n-n'} \lvert A \rvert
/n!$, where $c > 1$ is a constant related to the lack of pseudorandomness. To see that density increment gives
a handle on the size of $A$, consider the following scenario. Initially, if $A$ is not too small, say $\lvert
A \rvert > n!/c_0^n$ for some constant $c_0 < c$, then any $A' \subseteq A$ must be ultimately embedded in
$S_{n'}$ with degree $n' = \Omega_{c_0,c}(n)$, otherwise it is not difficult to show that the density of $A'$
in $S_{n'}$ would be larger than $1$ which is impossible.

We dealt with low complexity shapes above, but we need to explain how
to analyze the remaining shapes. Following a similar argument of KLR,
we show that provided the low complexity shapes are not too few, all
the remaining shapes can be absorbed in a ``tail bound'' which
crucially relies on $c_0 < 2$.

The steps above produce a family of linear programs with parameters $c$ and $\ell_0$ (but not depending on
$n$). If for a given choice of these parameters the associated linear program has optimal value $M > 0$,
then we conclude that $\alpha(\cB_n) \le K \cdot n!/(c-o(1))^n$ for some universal constant $K > 0$. We obtain
our main result, \cref{theo:main_ub_alpha}, by computationally solving a carefully chosen set of
parameters. Let us point out that solving these linear programs as $\ell_0$ gets larger and $c$ approaches $2$
becomes quite challenging even computationally; this requires additional ideas, which we discuss in
\cref{sec:computational}.

% LocalWords:  pseudorandom KLR pseudorandomness

\section{Preliminaries}\label{sec:prelim}

We denote the set of natural numbers by $\NN \coloneqq \{0,1,\ldots,\}$ and the set of positive integers by
$\NN_+ \coloneqq \NN\setminus\{0\}$. Given $n,k\in\NN_+$ with $n \ge k$, we let $[n] \coloneqq \{1,\ldots,n\}$
(and $[0]\coloneqq\varnothing$) and let $[n]_k \coloneqq \{(i_1,\ldots,i_k)\in [n]^k :
\lvert\{i_1,\ldots,i_k\}\rvert = k\}$ be the set of $k$-tuples of elements of $[n]$ with no repeated
coordinates. We also let $(n)_k \coloneqq n(n-1)\cdots(n-k+1)$ denote the falling factorial so that
$\lvert[n]_k\rvert = (n)_k$. Let $S_n$ be the symmetric group on $[n]$. We denote the sign of permutation
$\sigma$ by $\sgn(\sigma)$. Let $\cC_{n,\ell}\subseteq S_n$ be the set of all single cycles of length $\ell$
in $S_n$ and let $\cC_n \coloneqq \bigcup_{\ell=2}^n \cC_{n,\ell}$.

\begin{definition}[Birkhoff Graph]
  The \emph{Birkhoff Graph} $\cB_n$ is the Cayley graph $\Cay(S_n,\cC_n)$, i.e., the vertex set is $S_n$ and
  $\sigma,\tau \in S_n$ are adjacent if and only if $\sigma\tau^{-1} \in \cC_n$.
\end{definition}

\begin{remark}\label{remark:birkhoff_auto}
  Recall that since $\cC_n$ is closed under conjugation, for every $\sigma \in S_n$, both multiplication by
  $\sigma$ maps $\tau \mapsto \tau \sigma$ and $\tau \mapsto \sigma \tau$ are automorphisms of $\cB_n$.
\end{remark}

\subsection{Representation of the Symmetric Group}

We recall some important definitions and results from the representation theory of $S_n$. For a thorough
introduction, we point the reader to the book of Sagan~\cite{sagan13} (see~\cite{serre96} for an introduction
to general representation theory). The irreducible representations of $S_n$ are the so-called \emph{Specht
  modules}, which are in one-to-one correspondence with \emph{partitions} of $n$. Recall that a partition of
$n$, $\lambda \vdash n$, is a tuple $\lambda \coloneqq (\lambda_1,\ldots,\lambda_s)$ of positive integers with
$\lambda_1 \ge \cdots \ge \lambda_s$ and $\sum_{i=1}^s \lambda_i = n$; the length $s$ of $\lambda$ as a
sequence is called the \emph{height} of $\lambda$ and denoted by $\htt(\lambda)$ and the \emph{size} of
$\lambda$ is denoted by $\lvert\lambda\rvert\coloneqq n$. It will be convenient to
visualize a partition via Young diagram (also known as Ferrers diagram), which is a left-adjusted box diagram
in which the $i$th row has $\lambda_i$ boxes (see \cref{subfig:youngdiagram}).

We denote by $S^{\lambda}$ the Specht module corresponding to $\lambda \vdash n$ and by $\chi^{\lambda} \colon
S_n \to \RR$ its corresponding character. We let $f_{\lambda}$ be the dimension of
$S^{\lambda}$. Alternatively, $f_{\lambda}$ can be computed as $\chi^{\lambda}(\id_n)$, which also corresponds
to the number of \emph{standard tableaux} on shape $\lambda$ (see~\cite[Section~2.5]{sagan13}). A standard
tableau of shape $\lambda \vdash n$ is a filling of the Young diagram of $\lambda$, where each box is filled
with a distinct number from $[n]$ so that each row and each column is increasing (see
\cref{subfig:stdtableau}). More generally, a semi-standard tableau of shape $\lambda \vdash n$ and content
$\mu \vdash n$ is a filling of the Young diagram of $\lambda$ with $\mu_i$ copies of $i$ and which is
(strictly) increasing in each column and non-decreasing in each row (see \cref{subfig:semistdtableau}). The
number of tableaux of shape $\lambda$ and content $\mu$ is called the \emph{Kostka number} $K_{\lambda,\mu}$.

\begin{figure}[htbp]
  \input{young}
\end{figure}

Another family of important modules, also indexed by partitions $\mu \vdash n$, is that of the \emph{Young
  modules} $M^{\mu}$. We only consider Young modules associated with shapes of the form $h^n_k \coloneqq (n-k,
1^k)$ commonly referred to as \emph{hooks}. The module $M^{h^n_k}$ corresponds to the natural action of $S_n$
on $[n]_k$ given by $\sigma \cdot (i_1,\ldots,i_k) \coloneqq (\sigma(i_1),\ldots,\sigma(i_k))$. The
irreducible decomposition of $M^{\mu}$ is given by the Young's Rule, where the Kostka numbers give the
multiplicities.

\begin{theorem}[Young's Rule~\protect{\cite[Theorem~2.11.2]{sagan13}}]\label{theo:young_decomp}
 Let $\mu \vdash n$. We have
 \begin{align*}
   M^{\mu}
   & \simeq
   \bigoplus_{\lambda \vdash n}K_{\lambda,\mu} \cdot S^{\lambda}.
 \end{align*}
\end{theorem}

The main tool we use to compute characters of Specht modules is the so-called Murnaghan--Nakayama rule. First,
recall that characters are \emph{class functions}, i.e., functions on $S_n$ that are invariant under
conjugation. Since conjugacy classes of $S_n$ are in one-to-one correspondence with partitions of $n$ (the
conjugacy class of $\sigma \in S_n$ corresponds to the partition whose parts are the lengths of the cycles in
the cycle decomposition of $\sigma$), we typically view characters as functions defined on partitions of
$n$. To apply the Murnaghan--Nakayama rule, we will also need the notion of \emph{rim hook}. Recall that a rim
hook $\xi$ of a shape $\lambda \vdash n$ is a contiguous region in the (right) border of $\lambda$ that does
not contain a two-by-two sub-shape and whose removal leaves a valid shape denoted $\lambda \setminus \xi$ (see
\cref{fig:rim}).

\begin{figure}[htbp]
  \input{rim}
\end{figure}

\begin{theorem}[Murnaghan--Nakayama Rule~\cite{sagan13}]\label{theo:mn_rule}
  Let $\lambda \vdash n$ and $\mu \coloneqq (\mu_1,\ldots,\mu_k)$ be a partition of $n$. Then for every
  $i\in[k]$, we have
  \begin{align*}
    \chi^{\lambda}(\mu)
    & =
    \sum_{\xi} (-1)^{\htt(\xi)-1}
    \cdot \chi^{\lambda \setminus \xi}((\mu_1,\ldots,\widehat{\mu_i},\ldots,\mu_k)),
  \end{align*}
  where $\xi$ ranges over all the rim hooks of size $\mu_i$ in $\lambda$, $\htt(\xi)$ is the height of $\xi$
  (i.e., the number of rows of $\xi$) and $(\mu_1,\ldots,\widehat{\mu_i},\ldots,\mu_k)$ is partition of $n -
  \mu_i$ obtained from $\mu$ by omitting part $\mu_i$.
\end{theorem}

\begin{remark*}
  Note that the Murnaghan-Nakayama rule gives us the freedom to choose which part of $\mu$ to remove first. We
  will explore this flexibility in our proofs.
\end{remark*}

We will be working with elements in the group algebra $\RR[S_n]$, which are formal $\RR$-linear combinations
of elements in $S_n$. Given $\phi \coloneqq \sum_{\sigma \in S_n} \phi_{\sigma} \cdot \sigma \in \RR[S_n]$
with $\phi_{\sigma} \in \RR$, it is convenient to regard $\phi$ as a function $S_n \to \RR$ defined by $\sigma
\mapsto \phi_{\sigma}$. The space $\RR[S_n]$ is equipped with the inner product
\begin{align*}
  \langle \phi,\psi \rangle
  & \coloneqq
  \frac{1}{\lvert S_n \rvert} \sum_{\sigma \in S_n} \phi_{\sigma} \cdot \psi_{\sigma}.
\end{align*}

Under this inner product, the characters of irreducible representations form an orthonormal basis of the
sub-space of class functions \cite[Proposition~1.10.2]{sagan13}, from which we can extract the following
Parseval formula.

\begin{fact}[Parseval]\label{fact:parseval_sym}
  Let $\phi,\psi \in \RR[S_n]$ be class functions. Then
  \begin{align*}
    \langle \phi, \psi \rangle
    & =
    \frac{1}{\lvert S_n \rvert^2} \sum_{\lambda \vdash n} \chi^{\lambda}(\phi) \cdot \chi^{\lambda}(\psi),
  \end{align*}
  where $\chi^{\lambda}$ is linearly extended to class functions.
\end{fact}

By reflecting a shape $\lambda \vdash n$ along the diagonal, we obtain the transposed shape $\lambda^\top$
given by $\lambda^{\top}_i \coloneqq \lvert \{j : \lambda_j \ge i\} \rvert$. We end this section with a
simple, but useful fact of irreducible characters in $S_n$.

\begin{fact}\label{fact:sign_of_transpose}
  Let $\lambda \vdash n$ and $\pi \in S_n$. We have
  \begin{align*}
    \chi^{\lambda^{\top}}(\pi)
    & =
    \sgn(\pi) \cdot \chi^{\lambda}(\pi).
  \end{align*}
\end{fact}

\subsection{Recalling KLR Results}

In this section, we recall some key results of Kane, Lovett and Rao~\cite{lovett17}, with minor
generalizations when necessary. For the reader's convenience, the omitted proofs can be found in
\cref{app:klr_proofs}.

For every non-empty set $A \subseteq S_n$, define the class function $\phi_{A} \in \RR[S_n]$ as
\begin{align*}
  \phi_{A}
  & \coloneqq
  \frac{1}{|S_n|}\frac{1}{|A|^2} \sum_{\substack{\sigma \in S_n\\  \pi,\pi' \in A}} \sigma \pi (\pi')^{-1} \sigma^{-1}.
\end{align*}
For $\ell \in \{0,\ldots,n-2\}$, define another class function $\psi_{\ell} \in \RR[S_n]$ as
\begin{align*}
  \psi_{\ell}
  & \coloneqq
  \frac{1}{\lvert C_{n,n-\ell} \rvert} \sum_{\tau \in C_{n,n-\ell} } \tau,
\end{align*}
where (we recall that) $C_{n,n-\ell}$ is the set of $(n-\ell)$-cycles of $S_n$.

The following claim says that the number of edges of $\cB_n$ within $A$ corresponding to cycles of length
$n-\ell$ can be computed from the inner product of these class functions.

\begin{claim}\label{claim:parseval_to_edges}
  We have
  \begin{align*}
    \langle \phi_{A}, \psi_{\ell} \rangle
    & =
    \frac{\lvert E_{\ell}[A,A] \rvert}{\lvert A \rvert^2 \lvert S_n \rvert \lvert C_{n,n-\ell} \rvert },
  \end{align*}
  where
  \begin{align*}
    E_{\ell}[A,A]
    & \coloneqq
    \left\lbrace (\pi,\pi') \in A^2  : \pi (\pi')^{-1} \in C_{n,n-\ell}\right\rbrace.
  \end{align*}
  In particular, if $\sum_{\lambda \vdash n} \chi^{\lambda}(\phi_{A}) \chi^{\lambda}(\psi_{\ell}) > 0$, then
  $A$ contains an edge of $\cB_n$.
\end{claim}

\begin{proof}
  We have
  \begin{align*}
    \lvert E_{\ell}[A,A] \rvert
    & =
    \sum_{\substack{\pi,\pi' \in A}} \mathbbm{1}\left[\pi (\pi')^{-1} \in C_{n,n-\ell} \right]
    \\
    & =
    \frac{1}{\lvert S_n \rvert} \sum_{\sigma \in S_n} \sum_{\substack{\pi,\pi' \in A}}
    \mathbbm{1}\left[\sigma \pi (\pi')^{-1} \sigma^{-1} \in C_{n,n-\ell} \right]
    \\
    & =
    \lvert A \rvert^2 \lvert S_n \rvert \lvert C_{n,n-\ell} \rvert \langle \phi_{A}, \psi_{\ell} \rangle,
  \end{align*}
  where the second equality follows since $C_{n,n-\ell}$ is invariant under conjugation. Further, if we assume
  $\sum_{\lambda \vdash n} \chi^{\lambda}(\phi_{A}) \chi^{\lambda}(\psi_{\ell}) > 0$, then
  \cref{fact:parseval_sym} yields that $A$ contains an edge of $\cB_n$.
\end{proof}

The following two facts give basic properties about $\chi^{\lambda}(\phi_A)$.

\begin{restatable}[From KLR]{fact}{KlrNonNegChar}\label{fact:non_negativity_of_irrep}
  Let $\lambda \vdash n$. Then $\chi^{\lambda}(\phi_A) \ge 0$.
\end{restatable}

\begin{fact}[Character folding]\label{fact:folding_var_phi}
  Let $A \subseteq S_n$. If all permutations of $A$ have the same sign, then
  \begin{align*}
    \chi^{\lambda^{\top}}(\phi_A) & = \chi^{\lambda}(\phi_A).
  \end{align*}
\end{fact}

\begin{proof}
  It follows from \cref{fact:sign_of_transpose} and $\sgn(\sigma \pi (\pi')^{-1} \sigma^{-1}) = \sgn(\pi)
  \sgn(\pi') = 1$ for every $\pi,\pi' \in A$ and every $\sigma \in S_n$.
\end{proof}

To obtain an upper bound on $\chi^{\lambda}(\phi_A)$, we will use a notion of pseudorandomness for the set
$A$.

\begin{definition}[Pseudorandomness]
 Let $k \in [n]$ and $r > 0$. We say that a non-empty $A \subseteq S_n$ is \emph{$(k, r)$-pseudorandom} if for
 every $I,J \in [n]_{k}$, we have
 \begin{align*}
   \Pr_{\pi \in A}[\pi(J) = I] & < \frac{r}{(n)_k},
 \end{align*}
 where $\pi \in A$ is chosen uniformly at random.

 For $\theta \colon \NN_+ \to (0,\infty)$ a non-decreasing function, we say that $A \subseteq S_n$ is
 \emph{$\theta$-pseudorandom} if $A$ is $(k,\theta(k))$-pseudorandom for every \emph{even} $k \in [n]$.
\end{definition}

We will typically use functions of the form $k \mapsto c^k$ for some fixed $c > 1$ and in this case abuse
notation by saying $c^k$-pseudorandom.

The pseudorandomness condition implies the following upper bound on Young module characters.

\begin{restatable}[Implicit in KLR]{claim}{KlrYoundModuleBound}\label{claim:young_module_wp_bound}
  If $A$ is $(k,r)$-pseudorandom, then
  \begin{align*}
    \tr(M^{h^n_k}(\phi_{A})) & \le r.
  \end{align*}
\end{restatable}

An arbitrary non-trivial character can be bounded in terms of the pseudorandomness parameter and an
appropriate Kostka number as follows.

\begin{restatable}[Implicit in KLR]{lemma}{KlrIrrepCharBound}\label{lemma:phi_lambda_bound}
  Let $\lambda \vdash n$ be non-trivial (i.e., $\lambda \ne (1^n)$). If $A \subseteq S_n$ is
  $(k,r)$-pseudorandom and $K_{\lambda,h^n_k} \ne 0$, then
  \begin{align*}
    \chi^{\lambda}(\phi_A) & \le \frac{r-1}{K_{\lambda,h^n_k}}.
  \end{align*}
\end{restatable}

\section{Theoretical Proofs}\label{sec:theo_proofs}

Recall from \cref{sec:strategy} that our proof strategy is to use a density increment argument to construct
from a sufficiently large $A \subseteq S_n$ a pseudorandom set $C \subseteq S_{n'}$ containing at most as many
edges as $A$ in the Birkhoff graph $\cB_n$. Translating edge counting representation theoretic arguments to
linear programming, we will be able to deduce that pseudorandom sets are not independent, which in turn
implies that large $A \subseteq S_n$ are also not independent. Therefore, quantifying how large $A$ needs to
be to make this approach viable provides upper bounds on $\alpha(\cB_n)$. The density increment argument is
formalized in \cref{subsec:density_inc}, whereas linear programming arguments are formally treated in
\cref{subsec:lp}.

\subsection{Dichotomy: Structure vs Randomness}\label{subsec:density_inc}

We shed more light on the structure versus randomness dichotomy of~\cite{lovett17} by observing a density
increment phenomenon. For $A \subseteq S_n$, let $d_n(A) \coloneqq \lvert A \rvert / \lvert S_n \rvert$ be its
\emph{density}. The lemma below shows that if $A$ fails to be $(k,r)$-pseudorandom, then we can find a set $C$
in $S_{n-k}$ with larger density than $A$ and with at most as many edges as $A$.

\begin{lemma}[density increment step]\label{lemma:density_increment}
    If $A \subseteq S_n$ non-empty is not $(k,r)$-pseudorandom, then there exist $\sigma,\sigma' \in S_n$ and
    $B' \subseteq A$ such that $B \coloneqq \sigma B' \sigma'$ satisfies
    \begin{enumerate}
    \item For every $\tau \in B$ and every $i \in [n]\setminus[n-k]$, we have $\tau(i) = i$,
      \label{it:density_inc_1}
    \item $\lvert B \rvert \ge \lvert A \rvert\cdot r/(n)_k$.
      \label{it:density_inc_2}
    \end{enumerate}
    In particular, by letting $C \coloneqq \{\tau\vert_{[n-k]}\mid\tau \in B \} \subseteq S_{n-k}$, we have
    $d_{n-k}(C) \ge r \cdot d_n(A)$ and $\lvert E_{\cB_{n-k}}(C,C)\rvert \le \lvert E_{\cB_{n}}(A,A)\rvert$.
    Furthermore, if all permutations of $A$ have the same sign then all permutations of $C$ also have the same
    sign.
\end{lemma}

\begin{proof}
    Since $A$ is not $(k,r)$-pseudorandom, there exist $I,J \in [n]_k$ such that
    \begin{align*}
      \Pr_{\pi \in A}[\pi(J) = I] & \ge \frac{r}{(n)_k}.
    \end{align*}
    Let $B' \coloneqq \{\pi \in A \mid \pi(J) = I \}$. Take any $\sigma,\sigma' \in S_n$ such that
    $\sigma'(n-k+1,\ldots,n) = J$ and $\sigma(I) = (n-k+1,\ldots,n)$. Then $B\coloneqq \sigma B' \sigma'$
    satisfies~\ref{it:density_inc_1} and we have
    \begin{align*}
      \frac{r}{(n)_k}
      & \le
      \Pr_{\pi \in A}[\pi(J) = I]
      =
      \frac{\lvert B' \rvert}{\lvert A \rvert}
      =
      \frac{\lvert B \rvert}{\lvert A \rvert};
    \end{align*}
    thus item~\ref{it:density_inc_2} follows.

    Note that item~\ref{it:density_inc_1} implies that $C \subseteq S_{n-k}$ and
    \begin{align*}
      r \cdot d_n(A)
      & =
      \frac{r \cdot \lvert A \rvert}{\lvert S_n \rvert}
      =
      \frac{r \cdot \lvert A \rvert}{(n)_k \cdot \lvert S_{n-k}\rvert}
      \le
      \frac{\lvert B \rvert}{\lvert S_{n-k} \rvert}
      =
      \frac{\lvert C \rvert}{\lvert S_{n-k} \rvert}
      =
      d_{n-k}(C),
    \end{align*}
    i.e., $C$ is at least $r$ times denser than $A$. Furthermore, by \cref{remark:birkhoff_auto} we have
    $\lvert E_{\cB_{n}}(B,B) \rvert = \lvert E_{\cB_{n}}(B',B') \rvert \le \lvert E_{\cB_{n}}(A,A) \rvert$.
    The assertions about $C$ follow from the fact that the restriction map $f \colon \tau \mapsto
    \tau\vert_{[n-k]}$ is a bijection from $B$ to $C$ and if $\{\tau,\tau'\} \in E_{\cB_{n-k}}(C,C)$ then
    $\{f^{-1}(\tau),f^{-1}(\tau')\} \in E_{\cB_n}(B,B)$. Finally, if all permutations of $A$ have the same
    sign, say $s$, then trivially so do all permutations in $B'$. By construction, the sign of all
    permutations in $B$ and in $C$ is $\sgn(\sigma) \cdot s \cdot \sgn(\sigma')$.
\end{proof}

The following simple claim shows that we can pass from $A \subseteq S_n$ to $C \subseteq S_{n-k}$ without
decreasing the density. In particular, this allow us to adjust the degree $n'$ so that the $n'$ cycles are
even permutations.

\begin{corollary}[density preserval]\label{cor:density_preserval}
    If $A \subseteq S_n$ is non-empty and $k \in [n]$, then there exist $\sigma,\sigma' \in S_n$ and $B'
    \subseteq A$ such that $B\coloneqq\sigma B' \sigma'$ satisfies
    \begin{enumerate}
      \item For every $\tau \in B$ and every $i \in [n]\setminus[n-k]$, we have $\tau(i) = i$,
      \item $\lvert B \rvert \ge \lvert A \rvert/(n)_k$.
    \end{enumerate}
    In particular, by letting $C \coloneqq \{\tau\vert_{[n-k]}\mid\tau \in B \} \subseteq S_{n-k}$, we have
    $d_{n-k}(C) \ge d_n(A)$ and $\lvert E_{\cB_{n-k}}(C,C)\rvert \le \lvert
    E_{\cB_{n}}(A,A)\rvert$. Furthermore, if every permutation of $A$ has the same sign then every permutation
    of $C$ has the same sign.
\end{corollary}

\begin{proof}
  Follows from \cref{lemma:density_increment} by noting that every $A \subseteq S_n$ is not
  $(k,1)$-pseudorandom.
\end{proof}

\begin{lemma}[density increment]\label{lemma:full_density_increment}
  For every $c_0 > c \ge 1$, every $n \in \NN_+$ and $A \subseteq S_n$ with $d_n(A) \ge 1/c^n$, there exists a
  set $B \subseteq S_{m}$, where $m \ge (1 - \log_{c_0}(c)) n$ and $m \equiv n \pmod{2}$ such that
  \begin{enumerate}
    \item $B$ is $c_0^k$-pseudorandom,
    \item $\lvert E_{\cB_m}(B,B) \rvert \le \lvert E_{\cB_n}(A,A) \rvert$,
    \item $d_m(B) \ge d_n(A)$, and
    \item if all permutations in $A$ have the same sign, then all permutations in $B$
      have the same sign.
  \end{enumerate}
\end{lemma}

\begin{proof}
  We construct inductively $A_0 \subseteq S_{n_0}, A_1 \subseteq S_{n_1}, \ldots$ through the following
  algorithm.
  \begin{enumerate}[label={\arabic*.}]
  \item Set $n_0 \coloneqq n$ and $A_0 \coloneqq A$,
  \item Given $A_t \subseteq S_{n_t}$, if $A_t$ is $c_0^k$-pseudorandom, stop and set $T \coloneqq t$ and $B
    \coloneqq A_T$; otherwise, let $k_t$ be such that $A_t$ is not $(k_t,c_0^{k_t})$-pseudorandom with $k_t$
    even, let $n_{t+1} \coloneqq n_t - k_t$ and by \cref{lemma:density_increment}, let $A_{t+1} \subseteq
    S_{n_{t+1}}$ be such that
    \begin{enumerate}
    \item $d_{n_{t+1}}(A_{t+1}) \ge c_0^{k_{t}} \cdot d_{n_t}(A_t)$,
    \item $\lvert E_{\cB_{n_{t+1}}}(A_{t+1},A_{t+1})\rvert \le \lvert E_{\cB_{n_t}}(A_t,A_t) \rvert$,
    \item if all permutations in $A_t$ in all have the same sign, then all permutations in $A_{t+1}$ have the
      same sign.
    \end{enumerate}
  \end{enumerate}
  We claim that the above procedure stops. Indeed, using induction, for every $t$ such that $A_t$ is
  constructed we have
  \begin{align*}
    1
    & \ge
    d_{n_t}(A_t)
    \ge
    c_0^{\sum_{i=0}^{t-1} k_i} \cdot d_n(A)
    \ge
    \frac{c_0^{\sum_{i=0}^{t-1} k_i}}{c^n},
  \end{align*}
  which implies that $t \le \sum_{i=0}^{t-1} k_i \le n \log_{c_0}(c)$. Hence, the claim follows.

  Let $k \coloneqq \sum_{i=0}^{T-1} k_i$ and $m \coloneqq n_T = n - k$ so that $m \ge (1 - \log_{c_0}(c))n$.
  By construction, the $k_i$'s are necessarily even, from which $m \equiv n \pmod{2}$. Finally, observe that
  $\lvert E_{\cB_m}(B,B)\rvert = \lvert E_{\cB_m}(A_T,A_T) \rvert \le \cdots \le \lvert E_{\cB_n}(A_0,A_0)
  \rvert = \lvert E_{\cB_n}(A,A) \rvert$ and analogously we also have $d_{m}(B) \ge d_n(A)$ since $c_0 \ge
  1$. By induction, we also have that if all permutations of $A=A_0$ have the same sign, then all permutations
  of $B=A_T$ have the same sign.
\end{proof}

% LocalWords:  pseudorandom

\subsection{Linear Programs}\label{subsec:lp}

We define three (families of) linear programs in order to establish an upper bound on $\alpha(\cB_n)$. The
first is a convex relaxation closely capturing the representation theoretic argument for counting edges of
pseudorandom sets of $S_n$ while the last one does not depend on $n$ and captures the asymptotic properties of
the Birkhoff graph family. Each such linear program can be roughly described as follows.

\begin{enumerate}[wide,label={Linear Program~\Roman*.}]
\item It ``contains'' all pseudorandom sets $A \subseteq S_n$ as feasible points and its objective value being
  positive implies that pseudorandom sets are not independent. This allows us to deduce an upper bound on
  $\alpha(\cB_n)$ using the density increment results of \cref{subsec:density_inc}. Its number of variables,
  its coefficients and its number of constraints depend on $n$.
\item Its objective value being positive implies that the first linear program has positive objective value.
  Its number of variables is independent of $n$, but its coefficients and its number of constraints depend on
  $n$.
\item Similarly, its objective value being positive implies that the linear program II has positive objective
  value for every $n$ sufficiently large. This third linear program is completely independent of $n$.
\end{enumerate}

\subsubsection{Linear Program I}

In the definition below, we present a family of linear programs such that each $c^k$-pseudorandom $A \subseteq
S_n$ yields a feasible solution. If further the objective value of such solution is positive, then we will
show this implies that $A$ is not independent. These properties of this family of linear programs are
established in the next two lemmas.

%%%%%%%%%%%%%%%%%%%%%%%%%%%%%%%%%%%%%%%%%%%%%%%%%%%%%%%%
%                   LP 1
%%%%%%%%%%%%%%%%%%%%%%%%%%%%%%%%%%%%%%%%%%%%%%%%%%%%%%%%

\begin{definition}[Linear Program I]\label{def:lp_1}
  Given an odd positive integer $n$, a real $c > 1$ and a non-negative even integer $\ell_0 \le n$, let
  $P_n^{\ell_0}(c)$ be the following linear program.
  \begin{align}
    \text{minimize}\quad
    &
    M
    \notag
    \\
    \text{s.t.}\quad
    &
    M \ge \Psi_{\ell} & \forall \ell \le \ell_0 \text{ even},
    \label{eq:lp:m}
    \\
    \text{(Parseval)}\quad
    &
    \Psi_{\ell} =  \sum_{\lambda \vdash n} \chi^{\lambda}((n-\ell)) \cdot x_{\lambda}
    &
    \forall \ell \le \ell_0 \text{ even},
    \label{eq:lp:parseval}
    \\
    \text{(Young)}\quad
    &
    \sum_{\lambda \vdash n} K_{\lambda,h_m^n} \cdot x_{\lambda} \le c^m
    &
    \forall m \le n \text{ even},
    \label{eq:lp:young}
    \\
    \text{(Transposition)}\quad
    &
    x_{\lambda} = x_{\lambda^{\top}}
    &
    \forall \lambda \vdash n,
    \label{eq:lp:transposition}
    \\
    \text{(Unit)}\quad
    &
    x_{(n)} = x_{(1^n)} = 1,
    \label{eq:lp:unit_1}
    \\
    &
    x_{\lambda} \ge 0
    &
    \forall \lambda \vdash n,
    \label{eq:lp:nonnegative}
  \end{align}
  where $\chi^{\lambda}((n-\ell))$ stands for $\chi^{\lambda}$ evaluated at a cycle of length $n-\ell$ and the
  variables are $M$, $(\Psi_{\ell})_{\ell}$ and $(x_{\lambda})_{\lambda \vdash n}$.
\end{definition}

\begin{lemma}\label{lemma:pseudo_rand_char_are_feasible}
  Let $n$ be an odd positive integer and $c > 1$. If $A \subseteq S_n$ is $c^k$-pseudorandom and all
  permutations of $A$ have the same sign, then taking $x_{\lambda} \coloneqq \chi^{\lambda}(\phi_A)$, defining
  $\Psi_{\ell}$ by~\eqref{eq:lp:parseval} and letting $M \coloneqq \max_{\ell} \Psi_{\ell}$ gives a feasible
  solution $P_n^{\ell_0}(c)$.
\end{lemma}

\begin{proof}
  By definition, constraints~\eqref{eq:lp:m} and~\eqref{eq:lp:parseval} are trivially satisfied. Note that
  $x_{\lambda} \ge 0$ by \cref{fact:non_negativity_of_irrep}, so the constraints~\eqref{eq:lp:nonnegative} are
  also satisfied.

  Now, we proceed to show that constraints~\eqref{eq:lp:young} are also satisfied. Since $A$ is
  $c^m$-pseudorandom, \cref{claim:young_module_wp_bound} gives that $\tr(M^{h^n_m}(\phi_A)) \le c^m$ for every
  even $m \in [n]$. Combining with Young's rule, \cref{theo:young_decomp}, we have
  \begin{align*}
    c^m
    & \ge
    \tr(M^{h^n_m}(\phi_A))
    =
    \sum_{\lambda \vdash n} K_{\lambda,\mu} \cdot \chi^{\lambda}(\phi_A),
  \end{align*}
  showing that constraints~\eqref{eq:lp:young} are satisfied.

  Since all the permutations of $A$ have the same sign and $\chi^{\lambda}(\phi_A) = \sum_{\pi,\pi' \in A}
  \chi^{\lambda}(\pi (\pi')^{-1})/\lvert A\rvert^2$, the transposition constraint~\eqref{eq:lp:transposition}
  follows from \cref{fact:folding_var_phi}. Finally, note that $\chi^{(n)}(\phi_A) = 1$ and
  $\chi^{(1^n)}(\phi_A) = 1$, where the latter follows from the transposition constraint
  \eqref{eq:lp:transposition}.
\end{proof}

\begin{lemma}\label{lemma:original_lp_pos}
  Let $n$ be an odd positive integer, $c > 1$. Suppose $\OPT(P_n^{\ell_0}(c)) > 0$. If $A \subseteq S_n$ is
  $c^k$-pseudorandom and all permutations of $A$ have the same sign, then $A$ is not independent in $\cB_n$.
\end{lemma}

\begin{proof}
  By \cref{lemma:pseudo_rand_char_are_feasible}, setting $x_{\lambda} \coloneqq \chi^{\lambda}(\phi_A)$ and
  $\Psi_{\ell}$ and $M$ as in the lemma gives a feasible solution of $P_n^{\ell_0}(c)$, which must have a
  positive objective value since $\OPT(P_n^{\ell_0}(c)) > 0$. In particular, there exists an even integer
  $\ell \le \ell_0$ such that $\Psi_{\ell} > 0$, i.e.,
  \begin{align*}
    \sum_{\lambda \vdash n} \chi^{\lambda}(\phi_A) \cdot \chi^{\lambda}(\psi_{\ell}) & > 0.
  \end{align*}
  By \cref{claim:parseval_to_edges}, this implies that $A$ is not independent in $\cB_n$.
\end{proof}

Putting together the above two lemmas with the density increment results of \cref{subsec:density_inc}, we get
the following asymptotic upper bound on the independence number $\alpha(\cB_n)$.

\begin{proposition}\label{prop:up_bound_alpha_birk}
  Let $c_0 > 1$ and $\ell_0$ be a non-negative even integer. Suppose $n_0 \in \NN$ is such that for every $n
  \ge n_0$ odd, we have $\OPT(P_n^{\ell_0}(c_0)) > 0$. Then for every $c \in (1,c_0)$ and every integer $n
  \ge 1 + n_0/(1-\log_{c_0}(c))$,
  \begin{align*}
    \alpha(\cB_n) & \le 2\cdot \frac{n!}{c^{n-1}}.
  \end{align*}
\end{proposition}

\begin{proof}
  Suppose that $A_0 \subseteq S_n$ is an independent set in $\cB_n$ with $\lvert A_0 \rvert \ge 2\cdot
  n!/c^{n-1}$, i.e., $d_n(A_0) \ge 2/c^{n-1}$. Let $s$ be the most frequent permutation sign in $A_0$ and $A_1
  \coloneqq \{\sigma \in A_0 \mid \sgn(\sigma)=s\}$. If $n$ is odd, let $A_2 \coloneqq A_1$ and $n_2
  \coloneqq n$. Otherwise, let $n_2 \coloneqq n-1$ and apply \cref{cor:density_preserval} to obtain a $A_2
  \subseteq S_{n_2}$. By construction, we have $n_2 \ge n-1$, $A_2$ is an independent set of $\cB_{n_2}$, all
  permutations of $A_2$ have the same sign and $d_{n_2}(A_2) \ge d_{n}(A_0)/2 \ge 1/c^{n_2}$.

  By \cref{lemma:full_density_increment}, we can find $B \subseteq S_m$ with $m \ge (1-\log_{c_0}(c))n_2 \ge
  n_0$ and such that
  \begin{enumerate}
  \item $B$ is $c_0^k$-pseudorandom,
  \item $B$ is independent in $\cB_m$,
  \item $d_m(B) \ge d_{n_2}(A_2) \ge 1/c^{n_2}$ (in particular $B$ is non-empty), and
  \item all permutations in $B$ have the same sign.
  \end{enumerate}

  Since $m \ge n_0$, we have $\OPT(P_m^{\ell_0}(c_0)) > 0$. Therefore, by \cref{lemma:original_lp_pos} the
  set $B$ must have an edge in $\cB_m$ contradicting the assumption that $A_0$ is independent.
\end{proof}

\subsubsection{Linear Program II}

The next step is to define a second family of simpler linear programs with the number of variables being
independent of $n$. To do so we classify partitions based on their leg length and belly as defined below and
we use some basic properties of irreducible characters and Kostka constants to remove variables with large leg
length or large belly. We will show in \cref{prop:lp_2_pos_implies_lp_1_pos} that when $n$ is sufficiently
large, a positive optimum value in this family implies a positive optimum value for the preceding family
$P_n^{\ell_0}(c)$.

\begin{notation}
  We denote the hook of size $n$ and leg length $k$ by $h^n_k \coloneqq (n-k,1^k)$. More generally, if
  $\beta\coloneqq(\beta_1,\beta_2,\ldots,\beta_t)$ is a partition and $n \ge \lvert \beta \rvert + k + \beta_1
  + 1$, then let
  \begin{align*}
    b^n_{k,\beta} & \coloneqq (n-\lvert\beta\rvert-k, \beta_1+1,\beta_2+1,\ldots,\beta_t+1,1^{k-t}).
  \end{align*}
  In this case, $\beta$ is called the \emph{belly} of $b^n_{k,\beta}$ and $k$ is called the \emph{leg
    length}. See \cref{fig:belly_shape} for a pictorial image of $b^n_{k,\beta}$.
\end{notation}

\begin{remark*}
  Note that $(h^n_k)^\top = h^n_{n-k-1}$ and $(b^n_{k,\beta})^\top =
  b^n_{n-\lvert\beta\rvert-k-1,\beta^\top}$.
\end{remark*}

\begin{figure}[!htbp]
  \centering
  \begin{tikzpicture}[scale=0.5, every node/.style={scale=0.5, fill=white, inner sep=0}]
    \draw  (-2.5,2.5) rectangle (-2,-3);
    \draw  (-2.5,3) rectangle (6.5,2.5);
    \draw (-2,-1) -- (-0.5,-1) -- (-0.5,0) -- (0,0) -- (0,1.5) -- (1,1.5) -- (1,2.5);
    \node at (-1,1) {\huge $\beta$};
    \node at (-3,0) {\huge $k$};
  \end{tikzpicture}
  \caption{Pictorial image of the shape $b^n_{k,\beta}$.}
  \label{fig:belly_shape}
\end{figure}
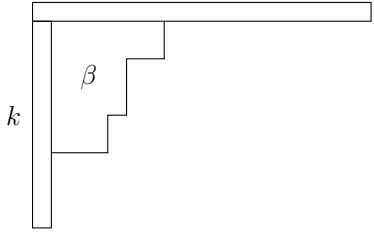

%%%%%%%%%%%%%%%%%%%%%%%%%%%%%%%%%%%%%%%%%%%%%%%%%%%%%%%%
%                   LP 2
%%%%%%%%%%%%%%%%%%%%%%%%%%%%%%%%%%%%%%%%%%%%%%%%%%%%%%%%

\begin{definition}[Linear Program II]\label{def:lp_2}
  Given a positive odd integer $n$, a real $c > 1$, a non-negative even integer $\ell_0 \le n$ and a positive odd
  integer $k_0$, we define $P_n^{\ell_0,k_0}(c)$ as follows.
  \begin{align}
    \text{minimize}\quad
    &
    M
    \notag
    \\
    \text{s.t.}\quad
    &
    M \ge \Psi_{\ell}
    &
    \forall \ell \le \ell_0 \text{ even},
    \label{eq:lp:m_trunc}
    \\
    \text{(Parseval)}\quad
    &
    \Psi_{\ell}
    =
    2 + 2 \cdot\sum_{i=0}^{\ell_0} \sum_{\beta \vdash i} \sum_{\substack{k=\htt(\beta)\\k\geq 1}}^{k_0}
    \chi^{b^n_{k,\beta}}((n-\ell)) \cdot x_{b^n_{k,\beta}}
    - 2 \cdot T_n^{\ell,k_0}(c)
    &
    \forall \ell \le \ell_0 \text{ even},
    \label{eq:lp:parseval_trunc}
    \\
    \text{(Young)}\quad
    &
    \sum_{i=0}^{\ell_0} \sum_{\beta \vdash i} \sum_{k = \htt(\beta)}^{k_0}
    K_{b^n_{k,\beta},h_m^n} \cdot x_{b^n_{k,\beta}}
    \le
    c^m
    &
    \forall m \le n \text{ even},
    \label{eq:lp:young_trunc}
    \\
    \text{(Unit)}\quad
    &
    x_{(n)} = 1,
    \label{eq:lp:unit_1_trunc}
    \\
    &
    x_{b^n_{k,\beta}} \ge 0
    &
    \mathllap{%
      \begin{aligned}[t]
        \forall i \in \{0,\ldots,\ell_0\},\\
        \forall \beta \vdash i,\\
        \forall k \in \{\htt(\beta),\ldots,k_0\},
      \end{aligned}%
    }
    \label{eq:lp:nonnegative_trunc}
  \end{align}
  where the variables are $M$, $(\Psi_{\ell})_{\ell}$ and $(x_{b^n_{k,\beta}})_{k,\beta}$ and we have
  \begin{align*}
    T_n^{\ell,k_0}(c)
    & \coloneqq
    \sum_{\substack{k=k_0 + 2\\ k \text{ odd}}}^{\frac{n-3\ell-1}{2}}
    \frac{c^{2k+2\ell}}{\binom{2k+2\ell}{k+\ell}}
    + \sum_{\substack{k=\frac{n-3\ell-1}{2} + 1\\ k \text{ odd}}}^{\frac{n-\ell-1}{2}}
    \frac{c^{n-1}}{\binom{n-\ell-1}{k+\ell}}.
  \end{align*}
\end{definition}

%%%%%%%%%%%%%%%%%%%%%%%%%%%%%%%%%%%%%%%%%%%%%%%%%%%%%%%%%%%%%%%%%%%%%%%%%%%%%%%%%%
%                     Parseval Coefficient Properties
%%%%%%%%%%%%%%%%%%%%%%%%%%%%%%%%%%%%%%%%%%%%%%%%%%%%%%%%%%%%%%%%%%%%%%%%%%%%%%%%%%

First, we show that if the belly size of a partition is larger than $\ell$, then its coefficient in the
Parseval~\eqref{eq:lp:parseval} is zero, so the corresponding variable in the linear program can be ignored.

\begin{claim}\label{claim:large_belly}
  If $\lvert \beta \rvert > \ell$, then $\chi^{b^n_{k,\beta}}((n-\ell)) = 0$.
\end{claim}

\begin{proof}
  We apply the Murnaghan--Nakayama rule, \cref{theo:mn_rule}, removing the $\ell$ fixed points first. Let
  $\lambda$ be the resulting shape after their removal in a particular derivation path. Note that $\lvert
  \lambda \rvert = n- \ell$ and we still need to remove a rim hook of size $n-\ell$. The contribution of the
  corresponding path is non-zero only if $\lambda$ is a hook. But for $\lambda$ to be a hook, the belly
  $\beta$ needs to be completely removed by some of the $\ell$ removed fixed points and thus $\lvert \beta
  \rvert \le \ell$.
\end{proof}

Next, we show that the coefficients $\chi^{b^n_{k,\beta}}((n-\ell))$ are still zero as long as the belly size
is smaller than $\ell$ but the leg length is not too small nor too large. It will be convenient to use the
following notation for two distinguished cells of a partition.
\begin{notation}
  Let $\lambda \vdash n$. We call \emph{hand} of $\lambda$ the rightmost cell in the first row of
  $\lambda$. We call \emph{foot} of $\lambda$ the lowest cell in the first column of $\lambda$.
\end{notation}

\begin{claim}\label{claim:thin_balanced}
  If $\lvert \beta \rvert < \ell \le n/2-1$ and $\ell \le k \le n - \lvert \beta \vert - \ell - 1$, then
  $\chi^{b^n_{k,\beta}}((n-\ell)) = 0$.
\end{claim}

\begin{proof}
  We apply the Murnaghan--Nakayama rule, \cref{theo:mn_rule}, removing $(n-\ell)$ first. To leave a valid
  shape after this removal, the following must hold:
  \begin{enumerate*}
  \item if a cell is removed in the first row of $b^n_{k,\beta}$, then all cells to its right must also be
    removed.
  \item if a cell is removed in the first column of $b^n_{k,\beta}$, then all cells below it must also be
    removed.
  \end{enumerate*}

  Our goal is to show that no removal leaves a valid shape. Now we have four cases to consider when we remove
  a valid rim hook of size $n-\ell$ from $b^n_{k,\beta}$.
  \begin{enumerate}[wide,label={Case \arabic*.}]
  \item The removal occurred completely inside the belly $\beta$. This implies $n-\ell \le \lvert \beta
    \rvert$. Since $\ell \le n/2-1$, this contradicts $\lvert \beta \rvert <\ell$.
  \item The hand and the foot were removed. In this case, the remaining shape must be $\beta$ so $\lvert \beta
    \rvert = \ell$. This contradicts the assumption $\lvert \beta \rvert < \ell$.
  \item The hand was removed but not the foot, which implies that $k + 1 \le \ell$ since no cell in the first
    column was removed. This contradicts the assumption $\ell \le k$.
  \item The foot was removed but not the hand, which implies that $n - \lvert \beta \rvert - k \le \ell$ since
    no cell in the first row was removed. This contradicts the assumption $k \le n -\lvert \beta \rvert -\ell
    -1$.
  \end{enumerate}
  Since no removal leaves a valid shape, we have $\chi^{b^n_{k,\beta}}((n-\ell)) = 0$.
\end{proof}

For the size of the belly being exactly $\ell$, the coefficient
$\chi^{b^n_{k,\beta}}((n-\ell))$ is not zero but can be computed with
the following claim.

\begin{claim}\label{claim:char_single_cycle}
  If $\beta \vdash \ell$, then
  \begin{align*}
    \chi^{b^n_{k,\beta}}((n-\ell)) & = (-1)^k f_{\beta}.
  \end{align*}
\end{claim}

\begin{proof}
  We apply the Murnaghan--Nakayama rule, \cref{theo:mn_rule}, removing $(n-\ell)$ first. Note the sign of the
  rim hook is $(-1)^k$ and the remaining shape is $\beta$ from which the claim readily follows.
\end{proof}

%%%%%%%%%%%%%%%%%%%%%%%%%%%%%%%%%%%%%%%%%%%%%%%%%%%%%%%%%%%%%%%%%%%%%%%%%%%%%%%%%
%                        Kostka Bounds
%%%%%%%%%%%%%%%%%%%%%%%%%%%%%%%%%%%%%%%%%%%%%%%%%%%%%%%%%%%%%%%%%%%%%%%%%%%%%%%%%

Observe that \cref{claim:char_single_cycle} says that a number of partitions depending on $n$ have non-zero
coefficient in the Parseval~\eqref{eq:lp:parseval}. To be able to remove most of the corresponding variables,
we will make use of the Young restrictions~\eqref{eq:lp:young}, which will require bounding the Kostka
constants. In one regime, we also compute the Kostka constant exactly since this will be used later. In the
process, we will need the following derived shape.

\begin{notation}
  Let $\beta\coloneqq(\beta_1,\ldots,\beta_t)$ be a shape and $k \ge t$ be an integer. We define
  $\mu_{k,\beta}\coloneqq (\beta_1+1,\beta_2+1,\ldots,\beta_t+1,1^{k-t})$ to be shape obtained from
  $b^n_{k,\beta}$ by removing the first row. Note that this does not depend on $n$.
\end{notation}

\begin{claim}\label{claim:kostka_belly}
  Let $\beta\coloneqq(\beta_1,\ldots,\beta_t)$ be a partition. If $m \le n-\beta_1-1$, then
  \begin{align*}
    K_{b^n_{k,\beta},h_{m}^n}
    & =
    \binom{m}{k+\lvert \beta \rvert} f_{\mu_{k,\beta}}
    \ge
    \binom{m}{k+\lvert \beta \rvert} f_{\beta}.
  \end{align*}
\end{claim}

\begin{proof}
  First, we prove the equality. Under the assumption $m \le n-\beta_1-1$, the palette $h_{m}^n$ has enough
  ones to fill the first $\beta_1 + 1$ positions of the first row of $b^n_{k,\beta}$ (and necessarily all the
  ones go in this first row). Then we can choose $k + \lvert \beta \rvert$ colors out of $\{2,\ldots,m+1\}$ to
  fill the positions \emph{not} in the first row of $b^n_{k,\beta}$ giving a total of $\binom{m}{k+\lvert
    \beta \rvert}$ possibilities. For each such possibility, there are exactly $f_{\mu_{k,\beta}}$ ways of
  filling these positions with these colors, since $\mu_{k,\beta}$ is their shape.

  The inequality follows by observing that the shape $\beta$ is contained in the shape $\mu_{k,\beta}$
  resulting in $f_{\mu_{k,\beta}} \ge f_{\beta}$.
\end{proof}

The following claim provides a defective bound when $m = n-1$ and the condition of \cref{claim:kostka_belly}
is not met.

\begin{claim}\label{claim:kostka_lb}
  Let $\beta\coloneqq(\beta_1,\ldots,\beta_t) \vdash \ell$. Then
  \begin{align*}
    K_{b^n_{k,\beta},h_{n-1}^n}
    & \ge
    \binom{n-\ell - 1}{k + \ell} f_{\mu_{k,\beta}}
    \ge
    \binom{n - \ell - 1}{k + \ell} f_{\beta}.
  \end{align*}
\end{claim}

\begin{proof}
  We start by the first inequality. Since we are only interested in a lower bound, we only consider the
  fillings which place the numbers in $[\beta_1+1]$ of the palette $h_{n-1}^n$ in the first $\beta_1+1$
  positions of the first row of $b^n_{k,\beta}$. Since $\lvert \beta \rvert = \ell$, this leaves at least
  $n-\ell-1$ colors out of which we choose to fill the remaining rows of $b^n_{k,\beta}$ accounting for
  $k+\ell$ cells. As before, for each such possibility, there are exactly $f_{\mu_{k,\beta}}$ ways of filling
  these positions with these colors, since $\mu_{k,\beta}$ is their shape and we have filled the first
  $\beta_1+1$ cells of the first row of $b^n_{k,\beta}$ with the numbers in $[\beta_1+1]$.

  The second inequality follows again by $f_{\mu_{k,\beta}} \ge f_{\beta}$ since the shape $\beta$ is
  contained in the shape $\mu_{k,\beta}$.
\end{proof}

%%%%%%%%%%%%%%%%%%%%%%%%%%%%%%%%%%%%%%%%%%%%%%%%%%%%%%%%%%%%%%%%%%
%                  Key Lemma LP 2
%%%%%%%%%%%%%%%%%%%%%%%%%%%%%%%%%%%%%%%%%%%%%%%%%%%%%%%%%%%%%%%%%%

Now, we relate the optimum values of the second linear program from \cref{def:lp_2} and the first one
from \cref{def:lp_1} for $n$ sufficiently large.

\begin{proposition}\label{prop:lp_2_pos_implies_lp_1_pos}
  Let $n$ be an odd positive integer, $c > 1$ and $\ell_0$ be non-negative even integer. Suppose $k_0$ is a
  positive integer such that $\ell_0 \le k_0 \le (n- \ell_0 - 3)/2$. Then
  \begin{align*}
    \OPT(P_n^{\ell_0,k_0}(c)) & \le \OPT(P_n^{\ell_0}(c)).
  \end{align*}
\end{proposition}

\begin{proof}
  Fix an optimum solution $(M, \Psi, x)$ to $P_n^{\ell_0}(c)$. Let $x'$ be the tuple obtained from $x$ by
  ignoring variables of the form $x_{b^n_{k,\beta}}$ where $k \ge k_0 + 1$. We compute $\Psi'$ from $x'$
  using~\eqref{eq:lp:parseval_trunc} and we let $M' \coloneqq \max_{\ell} \Psi_{\ell}'$.

  We claim that $(M',\Psi',x')$ is a feasible solution of $P_n^{\ell_0,k_0}(c)$ (not necessarily with the same
  value as $(M,\Psi,x)$). Trivially, restrictions~\eqref{eq:lp:m_trunc}, \eqref{eq:lp:parseval_trunc},
  \eqref{eq:lp:unit_1_trunc} and~\eqref{eq:lp:nonnegative_trunc} are satisfied. Also observe that
  restrictions~\eqref{eq:lp:young_trunc} follow from the fact that $x$ is non-negative and it satisfied the
  restrictions~\eqref{eq:lp:young} from $P_n^{\ell_0}(c)$.

  We will now prove that $M' \le M$. For this, it is enough to show that $\Psi'_{\ell} \le \Psi_{\ell}$ for
  every non-negative even $\ell \le \ell_0$. Fix some such $\ell$. Since $k_0 \le (n- \ell_0 - 3)/2$, if $k
  \le k_0$ and $\beta$ is a shape of size at most $\ell$, then $(b^n_{k,\beta})^\top = b^n_{n-\lvert \beta
    \rvert - k - 1,\beta^\top}$ and $n-\lvert \beta \rvert - k - 1 \ge k_0 +1$ implying that the variable
  $x_{(b^n_{k,\beta})^\top}$ is not in $x'$. Now note that by \cref{fact:folding_var_phi} we have
  $\chi^{b^n_{k,\beta}}((n-\ell)) = \chi^{(b^n_{k,\beta})^\top}((n-\ell))$ since $\sgn((n-\ell))$ is
  positive. By \cref{claim:large_belly}, we have $\chi^{b^n_{k,\beta}}((n-\ell)) = 0$ whenever $\lvert
  \beta\rvert > \ell$. On the other hand, when $\lvert \beta \rvert < \ell \le \ell_0$, since $k_0 \ge
  \ell_0$, by \cref{claim:thin_balanced} we have $\chi^{b^n_{k,\beta}}((n-\ell)) = 0$ for every integer $k$
  such that $k_0 +1 \le k \le n - \lvert \beta \vert - k_0 - 2$. From these and
  restriction~\eqref{eq:lp:transposition} ($x_{\lambda} = x_{\lambda^{\top}}$) from $P_n^{\ell_0}(c)$ and
  $x_{(n)}=x_{(1^n)}=1$, we conclude that
  \begin{align*}
    \Psi_{\ell} - \Psi'_{\ell}
    & =
    \sum_{\beta \vdash \ell } \sum_{k = k_0 + 1}^{n-i-k_0-2}
    \chi^{b^n_{k,\beta}}((n-\ell)) \cdot x_{b^n_{k,\beta}}
    + 2 \cdot T_n^{\ell,k_0}(c).
  \end{align*}
  Using \cref{claim:char_single_cycle}, we have
  \begin{equation}\label{eq:diff_psi_l}
    \begin{aligned}
      \Psi_{\ell} - \Psi'_{\ell}
      & =
      \sum_{\beta \vdash \ell } \sum_{k = k_0 + 1}^{n-\ell-k_0-2} (-1)^k f_{\beta} \cdot x_{b^n_{k,\beta}}
      + 2 \cdot T_n^{\ell,k_0}(c)
      \\
      & \ge
      - \sum_{\beta \vdash \ell } \sum_{\substack{k = k_0 + 1\\ k \text{ odd}}}^{n-\ell-k_0-2}
      f_{\beta} \cdot x_{b^n_{k,\beta}}
      + 2 \cdot T_n^{\ell,k_0}(c)
      \\
      & \ge
      - 2 \sum_{\beta \vdash \ell } \sum_{\substack{k = k_0 + 1 \\ k \text{ odd}}}^{(n-\ell-1)/2}
      f_{\beta} \cdot x_{b^n_{k,\beta}}
      + 2 \cdot T_n^{\ell,k_0}(c),
    \end{aligned}
  \end{equation}
  where the last inequality follows from~\eqref{eq:lp:transposition} ($x_{\lambda} = x_{\lambda^{\top}}$) and
  note that we are double counting the cases where $\lambda = \lambda^{\top}$, namely, when $k=(n-\ell-1)/2$.

  Fix an odd integer $k$ in $[k_0+1,n-\ell-k_0-2]$. We consider two cases.

  The first case is when $k \le (n-3\ell-1)/2$, where we use the Young restriction~\eqref{eq:lp:young_trunc}
  for $m=2(k+\ell)$ (and the fact that $x \ge 0$) to get that
  \begin{align*}
    \sum_{\beta \vdash \ell} K_{b^n_{k,\beta},h^n_{2(k+\ell)}} \cdot x_{b^n_{k,\beta}}
    & \le
    c^{2(k+\ell)}.
  \end{align*}
  By our choice of $k$ and by \cref{claim:kostka_belly}, we can simplify the above equation as
  \begin{equation}\label{eq:tail_bound_easy}
    \sum_{\beta \vdash \ell}  f_{\beta} \cdot x_{b^n_{k,\beta}} \le \frac{c^{2(k+\ell)}}{\binom{2(k+\ell)}{k+ \ell}}.
  \end{equation}

  The second case is when $k > (n-3\ell-1)/2$, where we use the Young restriction~\eqref{eq:lp:young_trunc}
  for $m=n-1$ (and the fact that $x \ge 0$) to get that
  \begin{align*}
    \sum_{\beta \vdash \ell} K_{b^n_{k,\beta},h^n_{n-1}} \cdot x_{b^n_{k,\beta}}
    & \le
    c^{n-1}.
  \end{align*}
  By \cref{claim:kostka_lb}, we can simplify the above equation as
  \begin{equation}\label{eq:tail_bound_hard}
    \sum_{\beta \vdash \ell}  f_{\beta} \cdot x_{b^n_{k,\beta}} \le \frac{c^{n-1}}{\binom{n-\ell-1}{k + \ell}}.
  \end{equation}

  Combining~\eqref{eq:diff_psi_l},~\eqref{eq:tail_bound_easy} and~\eqref{eq:tail_bound_hard}, we conclude that
  $\Psi_{\ell} - \Psi'_{\ell} \ge 0$ for every non-negative even $\ell \le \ell_0$ implying
  \begin{align*}
    \OPT(P_n^{\ell_0,k_0}(c))
    &\le
    M'
    \le
    M
    =
    \OPT(P_n^{\ell_0}(c)),
  \end{align*}
  which concludes the proof.
\end{proof}

\subsubsection{Linear Program III}

To define the third family of linear programs, we first show that the coefficients of the second family of
linear programs $P_n^{\ell_0,k_0}(c)$ for $\ell_0,k_0$ fixed become independent of $n$ as long as $n$ is
sufficiently large. The main ingredient for the stabilization of the Kostka coefficients is
\cref{claim:kostka_belly}, whereas for Parseval coefficients it will be convenient to work with the following
generalization of shape.

\begin{definition}
  Given a shape $\beta$, an integer $k\geq\htt(\beta)$ and an integer $\ell \ge \lvert \beta \rvert$, we let
  $\xi_{k,\beta,\ell}$ be the (not necessarily valid) shape obtained from $b^n_{k,\beta}$ by removing the rim
  hook of size $n-\ell$ that contains the hand of $b^n_{k,\beta}$, where $n \ge \lvert \beta \rvert + k +
  \beta_1 + 1$, and we let $t_{k,\beta,\ell}$ be the height of this removed rim (note that
  $\xi_{k,\beta,\ell}$ and $t_{k,\beta,\ell}$ do not depend on the choice of $n$). Note that if
  $\lvert\beta\vert = \ell$, then $\xi_{k,\beta,\ell} = \beta$ and $t_{k,\beta,\ell} = k+1$. See
  \cref{fig:derived_shape} for some examples.
\end{definition}

\begin{figure}[htbp]
  \input{derived_shape}
\end{figure}

\begin{lemma}\label{lemma:parseval_coeff_limit}
  Suppose $n \ge \lvert \beta \rvert + k + \beta_1 + 1$ (so that $b^n_{k,\beta}$ is well-defined) and $\ell
  \ge \lvert \beta \rvert$. If $n -\ell > \lvert \beta \rvert + k$, then
  \begin{align*}
    \chi^{b_{k,\beta}^n}((n-\ell))
    & =
    \begin{dcases*}
      0, & if $\xi_{k,\beta,\ell}$ is not a valid shape\\
      (-1)^{t_{k,\beta,\ell}-1} \cdot f_{\xi_{k,\beta,\ell}}, & otherwise.
    \end{dcases*}
  \end{align*}
  In particular, the value above does \emph{not} depend on $n$.
\end{lemma}

\begin{proof}
  We apply the Murnaghan--Nakayama rule, \cref{theo:mn_rule}, removing $(n-\ell)$ first. To leave a valid
  shape after this removal, the following must hold:
  \begin{enumerate*}
  \item if a cell is removed in the first row of $b^n_{k,\beta}$, then all cells to its right (including the
    hand) must also be removed.
  \item if a cell is removed in the first column of $b^n_{k,\beta}$, then all cells below it (including the
    foot) must also be removed.
  \end{enumerate*}

  If the hand is not removed, then nothing in the first row was removed. Since there are $k + \lvert \beta
  \rvert < n - \ell$ cells not in the first row, this case is impossible. Hence, the hand must be removed, the
  shape remaining after this removal is $\xi_{k,\beta,\ell}$ and the removed rim hook has height
  $t_{k,\beta,\ell}$. If $\xi_{k,\beta,\ell}$ is a valid shape, then the removal of $(n-\ell)$ gives a sign
  of $(-1)^{t_{k,\beta,\ell}-1}$ and the removal of the remaining $\ell$ fixed points gives a factor of
  $f_{\xi_{k,\beta,\ell}}$. If $\xi_{k,\beta,\ell}$ is not a valid shape, then $\chi^{b_{k,\beta}^n}((n-\ell))
  = 0$.
\end{proof}

Since $\lim_{n\to \infty} \chi^{b_{k,\beta}^n}((n-\ell))$ is well-defined, we can give a name to this limit.

\begin{definition}\label{def:limit_coeff}
  \cref{lemma:parseval_coeff_limit} above states that $\chi^{b_{k,\beta}^n}((n-\ell))$ does not depend on $n$
  as long as $n \ge \max\{\lvert \beta \rvert + k + \beta_1 + 1, \lvert \beta \rvert + k + \ell\}$, so we let
  $\chi^{b_{k,\beta}}_{\ell} \coloneqq \lim_{n\to \infty} \chi^{b_{k,\beta}^n}((n-\ell))$.
\end{definition}

%%%%%%%%%%%%%%%%%%%%%%%%%%%%%%%%%%%%%%%%%%%%%%%%%%%%%%%%
%                   LP 3
%%%%%%%%%%%%%%%%%%%%%%%%%%%%%%%%%%%%%%%%%%%%%%%%%%%%%%%%

The third family of linear programs, which is completely independent of $n$, is defined as follows.

\begin{definition}[Linear Program III]\label{def:lp_3}
  Given a positive odd integer $k_0$, a real $c \in (1,2)$, a non-negative even integer $\ell_0$ and a
  positive even integer $m_0$, we let $P^{\ell_0,k_0,m_0}(c)$ be the following linear program.
  \begin{align}
    \text{minimize}\quad
    &
    M
    \notag
    \\
    \text{s.t.}\quad
    &
    M \ge \Psi_{\ell}
    &
    \forall \ell \le \ell_0 \text{ even},
    \label{eq:lp:m_limit}
    \\
    \text{(Parseval)}\quad
    &
    \Psi_{\ell}
    =
    2 + 2 \cdot \sum_{i=0}^{\ell_0} \sum_{\beta \vdash i} \sum_{\substack{k=\htt(\beta)\\k\geq 1}}^{k_0}
    \chi^{b_{k,\beta}}_{\ell} \cdot x_{k,\beta}
    - 2 \cdot T^{\ell,k_0}(c)
    &
    \forall \ell \le \ell_0 \text{ even},
    \label{eq:lp:parseval_limit}
    \\
    \text{(Young)}\quad
    &
    \sum_{i=0}^{\ell_0} \sum_{\beta \vdash i} \sum_{\substack{k = \htt(\beta)\\k\geq 1}}^{k_0}
    \binom{m}{k+\lvert \beta \rvert} f_{\mu_{k,\beta}} \cdot x_{k,\beta}
    \le
    c^m - 1
    &
    \forall m \le m_0 \text{ even},
    \label{eq:lp:young_limit}
    \\
    &
    x_{k,\beta} \ge 0
    &
    \mathllap{%
      \begin{aligned}[t]
        \forall i \in \{0,\ldots,\ell_0\},\\
        \forall \beta \vdash i,\\
        \forall k \in \{\htt(\beta),\ldots,k_0\},
      \end{aligned}%
    }
    \notag
  \end{align}
  where the variables are $M$, $(\Psi_{\ell})_{\ell}$
  and $(x_{k,\beta})_{k,\beta}$ and we have
  \begin{align*}
    T^{\ell,k_0}(c)
    & \coloneqq
    1.5  \cdot
    \frac{%
      \left(2\left(k_0 + \ell \right) + 4 \right) \left(\frac{c}{2}\right)^{2(k_0 + \ell)+4}
      - 2(k_0 + \ell) \left(\frac{c}{2}\right)^{2(k_0+\ell)+8}
    }{%
      \left(1- \left(\frac{c}{2}\right)^4\right)^2
    }.
  \end{align*}
\end{definition}

We now connect the optimum objective values from the third family to the second family of linear programs.
\begin{proposition}\label{prop:lp_3_pos_then_lp_2_pos}
  Let $c \in (1,2)$ and $\ell_0$ be a non-negative even integer. Let $k_0$ be a positive odd integer such that
  $\ell_0 \le k_0$.

  If $\OPT(P^{\ell_0,k_0,m_0}(c)) > 0$, then there exists an integer $n_0$ large enough such that
  $\OPT(P_n^{\ell_0,k_0}(c)) > 0$ for every odd integer $n \ge n_0$.
\end{proposition}

\begin{proof}
  Fix a positive odd integer $n$. Suppose $s_n \coloneqq (M^n, (\Psi_{\ell}^n)_{\ell},
  (x_{b^n_{k,\beta}}^n)_{k,\beta})$ is an optimum solution of $P_n^{\ell_0,k_0}(c)$.

  We will construct a solution $\widehat{s}_n \coloneqq (\widehat{M}^n, (\widehat{\Psi}_{\ell}^n)_{\ell},
  (\widehat{x}^n_{k,\beta})_{k,\beta})$ of $P^{\ell_0,k_0,m_0}(c)$. Set $\widehat{x}^n_{k,\beta} \coloneqq
  x^n_{b^n_{k,\beta}}$. We compute $\widehat{\Psi}^n$ from $\widehat{x}^n$ using~\eqref{eq:lp:parseval_limit}
  and we let $\widehat{M}^n \coloneqq \max_{\ell} \widehat{\Psi}_{\ell}^n$.

  Using \cref{claim:kostka_belly}, we have that the restrictions~\eqref{eq:lp:young_limit} are exactly the
  same as the restrictions~\eqref{eq:lp:young_trunc} (since $f_{\mu_{0,()}} = x^n_{(1^n)} = 1$). Hence,
  $\widehat{s}_n$ is a feasible solution of $P^{\ell_0,k_0,m_0}(c)$.

  Now we compare the objective values $\widehat{M}^n$ and $M^n$. For this, it is enough to compare
  $\widehat{\Psi}^n_{\ell}$ and $\Psi^n_{\ell}$ for every non-negative even integer $\ell \le \ell_0$. By
  \cref{lemma:parseval_coeff_limit} and \cref{def:limit_coeff}, for $n\geq 2\ell_0 + k_0 + 1$, we have
  \begin{align}\label{eq:Psi_diff}
    \widehat{\Psi}^n_{\ell} - \Psi^n_{\ell}
    & =
    - 2 \cdot T^{\ell,k_0}(c) + 2 \cdot T_n^{\ell,k_0}(c).
  \end{align}
  A straightforward computation done in \cref{claim:tail_bounding_tail} (in \cref{app:tail_bounds})
  establishes that
  \begin{align}\label{eq:tail_limit_bound}
    \lim_{\substack{n\to \infty \\ \text{odd}}} T_n^{\ell,k_0}(c)
    & \le
    T^{\ell,k_0}(c).
  \end{align}

  For every positive odd integer $n$, let $\widehat{\ell}_n,\ell_n$ be such that
  \begin{align*}
    \widehat{\Psi}^n_{\widehat{\ell}_n}
    & =
    \max_{\substack{\ell \le \ell_0\\\ell\text{ even}}}\widehat{\Psi}^n_{\ell},
    &
    \Psi^n_{\ell_n}
    & =
    \max_{\substack{\ell \le \ell_0\\\ell\text{ even}}} \Psi^n_{\ell},
  \end{align*}
  then we have
  \begin{align*}
    \widehat{M}^n - M^n
    & =
    \widehat{\Psi}^n_{\widehat{\ell}_n} - \Psi^n_{\ell_n}
    =
    \widehat{\Psi}^n_{\widehat{\ell}_n} - \Psi^n_{\widehat{\ell}_n}
    + \Psi^n_{\widehat{\ell}_n} - \Psi^n_{\ell_n}
    \leq
    \widehat{\Psi}^n_{\widehat{\ell}_n} - \Psi^n_{\widehat{\ell}_n},
  \end{align*}
  hence
  \begin{align*}
  \limsup_{\substack{n\to\infty\\\text{odd}}}(\widehat{M}^n - M^n)
  \le
  \limsup_{\substack{n\to\infty\\\text{odd}}}
  (\widehat{\Psi}^n_{\widehat{\ell}_n} - \Psi^n_{\widehat{\ell}_n})
  \le
  0,
  \end{align*}
  where the last inequality follows from~\eqref{eq:Psi_diff} and~\eqref{eq:tail_limit_bound}. Therefore, we
  obtain
  \begin{align*}
    \OPT(P^{\ell_0,k_0,m_0}(c))
    & \le
    \liminf_{\substack{n \to \infty \\\text{odd}}} \widehat{M}^n
    \leq
    \liminf_{\substack{n \to \infty \\ \text{odd}}} M^n
    =
    \liminf_{\substack{n \to \infty \\ \text{odd}}} \OPT(P^{\ell_0,k_0}_n(c)),
  \end{align*}
  which implies the result.
\end{proof}

The results of this section culminate in the next theorem, which reduces the problem of asymptotically upper
bounding $\alpha(\cB_n)$ to showing that a \emph{single} linear program from the third family \cref{def:lp_3}
has a positive objective value for a given $c \in (1,2)$.

\begin{theorem}\label{theo:lp_main}
  Let $c_0 > 1$, $\ell_0$ be a non-negative even integer, $k_0$ be a positive odd integer and $m_0$ be a
  positive even integer. Suppose we have $\OPT(P^{\ell_0,k_0,m_0}(c_0)) > 0$. Then for every $c \in (1,c_0)$,
  there exists $n_0\coloneqq n_0(c_0,c,\ell_0,k_0,m_0) \in \NN$ such that for every integer $n \ge n_0$
  \begin{align*}
    \alpha(\cB_n) & \le 2\cdot \frac{n!}{c^{n-1}}.
  \end{align*}
\end{theorem}

\begin{proof}
  Given a linear program $P^{\ell_0,k_0,m_0}(c_0)$ in the third family such that
  $\OPT(P^{\ell_0,k_0,m_0}(c_0)) > 0$, by \cref{prop:lp_3_pos_then_lp_2_pos} the linear program
  $P_n^{\ell_0,k_0}(c_0)$ in the second family satisfies $\OPT(P_n^{\ell_0,k_0}(c_0)) > 0$ for $n$
  sufficiently large. By \cref{prop:lp_2_pos_implies_lp_1_pos}, this in turn implies that the linear program
  $P_n^{\ell_0}(c_0)$ in the first family satisfies $\OPT(P_n^{\ell_0}(c_0)) > 0$ for $n$ sufficiently large.
  Finally, the bound on $\alpha(\cB_n)$ follows from \cref{prop:up_bound_alpha_birk}.
\end{proof}

% LocalWords:  pseudorandom Birkhoff

\section{Computational Part}\label{sec:computational}

From \cref{theo:lp_main}, to claim an asymptotic upper bound of $\alpha(\cB_n) \le K \cdot n!/(c-o_n(1))^n$
with $K =2c$, it is enough to find a single choice of parameters $\ell_0$, $k_0$ and $m_0$ that makes the
linear program $P^{\ell_0,k_0,m_0}(c)$ (from \cref{def:lp_3}) have positive optimum value. However, a full
theoretical analysis of these linear programs remains elusive. Nevertheless, we can still computationally
solve these programs for various specific choices of parameters. The closer the target $c$ is to $2$, the
larger the parameter $\ell_0$ (and $k_0$) needs to be to yield a positive optimum value. This computational
approach poses its own challenges and getting $c = 1.971$ requires careful consideration on how to solve these
linear programs.

The first challenge is the poor dependence of the size of these programs on the parameter $\ell_0$, both in
terms of number of variables and bit complexity of coefficients. More specifically, the number of variables
grows at least as fast as the number of partitions of $\ell_0$, which is asymptotically
\begin{align*}
  \frac{1}{4 \ell_0 \sqrt{3}} \cdot \exp\left(\pi \sqrt{\frac{2 \ell_0}{3}}\right),
\end{align*}
by a celebrated theorem of Hardy--Ramanujan~\cite{HardyR18}.

The second challenge is that the linear programs $P^{\ell_0,k_0,m_0}(c)$ seem to be very sensitive to
numerical rounding errors preventing the use of conventional LP solvers. We believe that such sensitivity
comes from the large difference in magnitude of the linear program coefficients (e.g., Kostka constants, see
\cref{claim:kostka_belly}). To avoid approximation errors and obtain an exact optimum solution, we implemented
the Simplex method with support to exact rational computations (note that this is enough since
$P^{\ell_0,k_0,m_0}(c)$ has rational coefficients as long as $c \in \mathbb{Q}$). Our implementation is
available at
\href{https://github.com/lenacore/birkhoff_code}{\url{https://github.com/lenacore/birkhoff_code}}.

In this section, we address these challenges.

\subsection{Dual Linear Program}

We actually solve the dual of the third linear program \cref{def:lp_3}, which is presented in
\cref{def:dual}. By strong linear programming duality, the dual program has positive optimum value if and only
if the primal has positive optimum value, so there is no loss in working with the dual program.

\begin{definition}[Dual Linear Program]\label{def:dual}
  Given a positive odd integer $k_0$, $c \in (1,2)$, a non-negative even integer $\ell_0$ and a positive even
  integer $m_0$, we let $D^{\ell_0,k_0,m_0}(c)$ be the following linear program.
  \begin{align}
    \text{maximize}\quad
    &
    2
    - 2\cdot \sum_{\substack{\ell=0\\ \ell\text{ even}}}^{\ell_0} T^{\ell,k_0}(c) \cdot w_{\ell}
    - \sum_{\substack{m \le m_0\\ m\text{ even}}} (c^m-1) \cdot y_m
    \notag
    \\
    \text{s.t.}\quad
    &
    \sum_{\substack{\ell=0\\ \ell\text{ even}}}^{\ell_0} w_{\ell} = 1
    \notag
    \\
    \text{(Restriction $b_{k,\beta}$)}\quad
    &
    2\cdot \sum_{\substack{\ell=0\\ \ell\text{ even}}}^{\ell_0} \chi^{b_{k,\beta}}_{\ell} \cdot w_{\ell}
    + \sum_{\substack{m \le m_0\\ m\text{ even}}} \binom{m}{k+\lvert \beta \rvert} f_{\mu_{k,\beta}} \cdot y_m \ge
    0
    &
    \begin{aligned}[t]
      \forall i \in \{0,\ldots,\ell_0\},\\
      \forall \beta \vdash i,\\
      \forall k \in [k_0], k \ge \htt(\beta),
    \end{aligned}
    \label{ineq:dual:partition}
    \\
    &
    w_{\ell} \ge 0
    &
    \forall \ell \le \ell_0 \text{ even},
    \notag
    \\
    &
    y_m \ge 0
    &
    \forall m \le m_0 \text{ even},
    \notag
  \end{align}
  where the variables are $(w_\ell)_\ell$ and $(y_m)_m$ and  $T^{\ell,k_0}(c)$ is as in \cref{def:lp_3}.
\end{definition}

There are two reasons for working with the dual linear program. First, we will be able to replace several
inequalities corresponding to shapes $b_{k,\beta}$ with large leg $k$ by a few provably more stringent
inequalities, an approach that we dub ``\emph{joint large leg}'' and is carried out in
\cref{subsec:joint_large_leg}. The second reason is due to a heuristic to speed the computation, called
``\emph{fragmented heuristic}'', which is explained in \cref{subsec:speeding_comp}.

\subsection{Joint Large Leg}\label{subsec:joint_large_leg}

To reduce the number of restrictions of the dual linear program of \cref{def:dual}, for a given $k \ge \ell_0$
and $s$ we replace all restrictions associated to partitions $b_{k,\beta}$, where $\lvert \beta \rvert = s$,
with a single more stringent restriction. Note that this modification can only decrease the objective value,
which still allows us to deduce asymptotic upper bounds on $\alpha(\cB_n)$ using \cref{theo:lp_main}.

\begin{lemma}
  Let $k \ge \ell_0 \ge s$ be non-negative integers with $k \ge 1$ and $(y_m)_m$ be non-negative. The inequality
  \begin{align*}
    2\cdot (-1)^k \cdot w_s + \sum_{\substack{m \le m_0\\ m\text{ even}}} \binom{m}{k+s}  \cdot y_m & \ge 0
  \end{align*}
  implies that the inequalities~\eqref{ineq:dual:partition} associated with $b_{k,\beta}$ in \cref{def:dual} for every $\beta \vdash s$
  are satisfied, i.e.,
  \begin{equation}\label{eq:dual_part_ineq}
    2\cdot \sum_{\substack{\ell=0\\ \ell\text{ even}}}^{\ell_0} \chi^{b_{k,\beta}}_{\ell} \cdot w_{\ell}
    + \sum_{\substack{m \le m_0\\ m\text{ even}}} \binom{m}{k+\lvert \beta \rvert} f_{\mu_{k,\beta}} \cdot y_m \ge 0.
  \end{equation}
\end{lemma}

\begin{proof}
  Since $k \ge \ell_0 \ge \ell$, to obtain $\xi_{k,\beta,\ell}$ from some $b_{k,\beta}^n$ we must have removed
  some cell in the first column; thus $\xi_{k,\beta,\ell}$ is a valid shape if and only if the foot of
  $b_{k,\beta}^n$ was removed, which in turn is equivalent to $\lvert \beta \rvert =\ell$. By
  \cref{lemma:parseval_coeff_limit}, inequality~\eqref{eq:dual_part_ineq} becomes
  \begin{align*}
    2\cdot (-1)^k \cdot f_{\beta} \cdot w_s
    +
    \sum_{\substack{m \le m_0\\ m\text{ even}}}\binom{m}{k+\lvert \beta \rvert} f_{\mu_{k,\beta}} \cdot y_m
    & \ge
    0.
  \end{align*}
  Then the result follows by noticing that the $y_m$ are non-negative and $f_{\mu_{k,\beta}} \ge f_{\beta}$
  since $\beta$ is contained in $\mu_{k,\beta}$.
\end{proof}

\subsection{Speeding the Computation}\label{subsec:speeding_comp}

We briefly explain three heuristics used to speed the computations. We stress that with any combination of
these heuristics if the objective value of the resulting linear program is positive, then the objective value
of the original dual linear program (\cref{def:dual}) is also positive.

\begin{itemize}
\item The first heuristic consists in setting some $y_m$ to $0$ (this corresponds to dropping restrictions
  in the primal) to decrease the size of the problem. Note that this can only decrease the optimum value.
\item To reduce the bit complexity of the program, we round up $(c^m-1)$ in the objective and we round down
  the Kostka constant $\binom{m}{k+\lvert \beta \rvert} f_{\mu_{k,\beta}}$. Similarly, this modification can
  only decrease the optimum value.
\item The final heuristic consists in solving a small ``fragment'' of the dual linear program containing much
  fewer restrictions. Of course, this can increase the optimum value. However, we can then check if an optimum
  solution to this fragment problem is feasible (therefore optimum) for original dual linear program. For
  reference, our best result used a fragment containing only restrictions associated with partitions
  $b^n_{k,\beta}$ for $\beta$ having height at most $1$ or being the partition $(1,1)$. This suggests that
  some optimum solutions of the dual program might have enough structure to be analyzable completely
  symbolically.
\end{itemize}

\subsection{Computational Results}

We finish this section by presenting some computational results in \cref{table:empirical_runs}, which contains
some parameters for which the dual linear program has positive optimum value.

\begin{table}[htb]
  \begin{center}
  \begin{tabular}{*{4}{>{$}c<{$}}}
    \ell_0 & c & k_0\\\hline
    0  & 1.49  & 19\\
    2  & 1.69  & 29\\
    4  & 1.72  & 29\\
    6  & 1.78  & 39\\
    8  & 1.80  & 39\\
    10 & 1.82  & 49\\
    12 & 1.85  & 59\\
    14 & 1.87  & 79\\
    % above this line no heuristic
    20 & 1.90  & 89  & * \\
    30 & 1.93  & 139 & *\\
    % above this line no heuristic except frag
    50 & 1.95  & 199 & **\\
    70 & 1.97  & 539 & **\\ % may need to be updated
    74 & 1.971 & 469 & **
  \end{tabular}
  \caption{List of parameters that yield positive optimum values for $D^{\ell_0,k_0,m_0}(c)$ (in all cases, we
    take $m_0 = 2(\ell_0+k_0)$). Entries marked with $*$ were computed using the fragment heuristic. Entries
    marked with $**$ were computed using all three heuristics and joint large leg. For reference, the instance
    with $\ell_0=74$ has approximately $3\times 10^{9}$ restrictions even after the joint large leg heuristic
    (before, the number was approximately $2.4 \times 10^{10}$), whereas the number of restrictions for
    $\ell_0=14$ and $\ell_0=30$ are approximately $3.8 \times 10^4$ and $3 \times 10^6$,
    respectively.}\label{table:empirical_runs}
  \end{center}
\end{table}

We remark that solving $P^{\ell_0,k_0,m_0}(c)$ (or $D^{\ell_0,k_0,m_0}(c)$) with $\ell_0=0$ can be viewed as
(essentially) the Kane--Lovett--Rao approach~\cite{lovett17}. In this case, we obtain an improved $c = 1.49$
over $c = \sqrt{2}$ from~\cite{lovett17} since we work with slightly stronger inequalities. It is interesting
to see that for $\ell_0=0$ making $k_0 > 19$ does not allow us to obtain a larger $c$ for which the dual has
positive optimum value. This means that increasing $\ell_0$ is crucial to obtain better values of $c$.

Combining the theoretical results from \cref{subsec:lp} and the computational results of this section, we
obtain our main result.

\TheoMainIndUB*

\begin{proof}
  Follows from \cref{theo:lp_main} and $\OPT(D^{\ell_0,k_0,m_0}(1.971)) > 0$ for $\ell_0 =74$, $k_0=469$ and
  $m_0=1086$.
\end{proof}

\section{Explicit Constructions}\label{sec:construction}

In this section we provide explicit constructions of independent sets and proper colorings of the Birkhoff
graph. Although~\cite{lovett17} only constructs independent sets and only when $n$ is a power of $2$, our
constructions build on similar ideas. However, since we adopt a simpler group theoretical language, this
enables us to achieve a modest improvement of an $n/2$ factor whenever $n$ is a power of $2$. Even though an
explicit independent set achieving the same bound can be deduced from our coloring, we first directly present
an independent set construction as it is simpler and serves as a warm up for the coloring construction.

\subsection{Independent Set}

We start by presenting in \cref{lemma:indep_set_basic} the recursive step of a construction of an independent
set that works in any size $n$. Such construction step will later be improved in
\cref{lemma:indep_set_improved} by a factor of $2$ conditioned on $n$ being divisible by $4$.

\begin{lemma}\label{lemma:indep_set_basic}
  Let $n \ge 2$ be an integer and suppose $A$ is an independent set of $\cB_{\lceil n/2 \rceil}$. Then the
  following is an independent set of $\cB_n$ (under the natural inclusion of $S_{\lfloor n/2 \rfloor} \times
  S_{\lceil n/2 \rceil}$ in $S_n$)
  \begin{align*}
    A'
    & \coloneqq
    \{ (\sigma, (\sigma')^{-1} \tau) \in S_{\lfloor n/2\rfloor}\times S_{\lceil n/2\rceil}\mid
    \sigma \in S_{\lfloor n/2 \rfloor}, \tau \in A \},
  \end{align*}
  where $\sigma'$ is the natural extension of $\sigma$ to $[\lceil n/2 \rceil]$ (by possibly fixing $\lceil
  n/2 \rceil$). In particular, we have
  \begin{align*}
    \lvert A' \rvert & = \left\lfloor \frac{n}{2} \right\rfloor! \cdot \lvert A \rvert.
  \end{align*}
\end{lemma}

\begin{proof}
  Note that in $S_{\lfloor n/2\rfloor}\times S_{\lceil n/2\rceil}$ a single cycle must either act only on the
  first part or only on the second part. This means that if $(\sigma_1, (\sigma'_1)^{-1} \tau_1), (\sigma_2,
  (\sigma'_2)^{-1} \tau_2) \in A'$ are adjacent in $\cB_n$, then either $\sigma_1 = \sigma_2$ and
  $(\sigma_1')^{-1} \tau_1 \cdot \tau_2^{-1} \sigma_2'$ is a single cycle; or $\sigma_1\cdot\sigma_2^{-1}$ is
  a single cycle and $(\sigma_1')^{-1}\tau_1 = (\sigma_2')^{-1}\tau_2$.

  In the first case, since we also have $\sigma_1' = \sigma_2'$, it follows that $\tau_1\tau_2^{-1}$ must also
  be a single cycle, contradicting the assumption that $A$ is independent in $\cB_{\ceil{n/2}}$. In the second
  case, we have $\tau_1\cdot\tau_2^{-1} = \sigma_1'\cdot(\sigma_2')^{-1}$, which must be a single cycle (as
  $\sigma_1\cdot\sigma_2^{-1}$ is so), generating the same contradiction.
\end{proof}

\begin{lemma}\label{lemma:indep_set_improved}
  Let $n \ge 4$ be an integer divisible by $4$, let $\gamma$ be the product of transpositions $\gamma\coloneqq
  \prod_{i=1}^{n/2} (i, n/2 + i)$ and suppose $A$ is an independent set of $\cB_{n/2}$ containing only
  permutations of positive sign. Let
  \begin{align*}
    A'
    & \coloneqq
    \{ (\sigma, \sigma^{-1} \tau) \in S_{n/2}\times S_{n/2}\mid
    \sigma \in S_{n/2}, \tau \in A \}.
  \end{align*}
  Then $A' \cup \gamma A'$ is an independent set of $\cB_n$ (under the natural inclusion of $S_{n/2} \times
  S_{n/2}$ in $S_n$) containing only permutations of positive sign. In particular, we have
  \begin{align*}
    \lvert A'\cup \gamma A' \rvert & = 2\left(\frac{n}{2}\right)! \cdot \lvert A \rvert.
  \end{align*}
\end{lemma}

\begin{proof}
  Since $n$ is divisible by $4$, we have $\sgn(\gamma) = 1$, so all permutations of $A'\cup\gamma A'$ have
  positive sign. By \cref{lemma:indep_set_basic}, we know that $A'$ is an independent set of $\cB_n$ and since
  $\pi\mapsto\gamma\pi$ is an automorphism of $\cB_n$, it follows that $\gamma A'$ is also an independent set
  of $\cB_n$.

  This means that if $A'\cup\gamma A'$ is not independent in $\cB_n$, it must contain an edge between some
  $\gamma\cdot(\sigma_1,\sigma_1^{-1}\tau_1)\in \gamma A'$ and some $(\sigma_2,\sigma_2^{-1}\tau_2)\in A'$,
  that is, the permutation
  \begin{align*}
    \pi
    & \coloneqq
    \gamma\cdot (\sigma_1\sigma_2^{-1},\sigma_1^{-1}\tau_1\tau_2^{-1}\sigma_2)
  \end{align*}
  must be a single cycle. But note that from the definition of $\gamma$, the permutation $\pi$ cannot have any
  fixed points, so $\pi$ must be a full cycle, which in particular implies that $\sgn(\pi) = -1$ (as $n$ is
  even). But this contradicts the fact that $\gamma\cdot(\sigma_1,\sigma_1\tau_1)$ and
  $(\sigma_2,\sigma_2\tau_2)$ both have positive sign.

  Note that $\gamma A'$ is contained in the left coset $\gamma (S_{n/2}\times S_{n/2})$, so it must be
  disjoint from $A'\subseteq S_{n/2}\times S_{n/2}$.
\end{proof}

Equipped with these two lemmas, we can now prove \cref{theo:indep_set_construct} (restated below). When $n$ is
a power of $2$, the factor of $2$ advantage of \cref{lemma:indep_set_improved} will compound to a total
advantage of $n/2$ in the final construction.

\TheoMainIndepConst*

\begin{proof}
  The first part of the theorem follows by a simple induction in $n$ using \cref{lemma:indep_set_basic} (the
  base case of $n=1$ consists of an independent set of size $1$ in $\cB_1$). The second part follows by
  induction in $\log_2(n)$ using \cref{lemma:indep_set_improved} instead and base cases of $n=1$ and $n=2$, in
  which the independent sets have size $1$.
\end{proof}

\subsection{Coloring}

Just as in the case of the independent set, we start by presenting in \cref{lemma:chi_basic} the recursive
step of a construction that works in any size $n$ and later improve this construction in
\cref{lemma:chi_improved} by a factor of $2$ when $n$ is divisible by $4$.

\begin{lemma}\label{lemma:chi_basic}
  Let $n \ge 2$ be an integer and suppose $f \colon S_{\lceil n/2 \rceil} \to \cX$ is a proper coloring of
  $\cB_{\lceil n/2 \rceil}$. Then there exists an explicit proper coloring of $\cB_n$ with
  \begin{align*}
    \binom{n}{\lceil n/2 \rceil} \cdot \lvert \cX \rvert
  \end{align*}
  colors.
\end{lemma}

\begin{proof}
  Set $A \coloneqq [\lceil n/2 \rceil]$ and $B \coloneqq [n]\setminus A$. Let $H \coloneqq S_A \times S_B
  \subseteq S_n$ and let $\cT \coloneqq \{t_1,\ldots, t_k\}$ be a set of representatives of (left) cosets of
  $H$ in $S_n$. Note that $k = \binom{n}{\lceil n/2 \rceil}$. Let $\iota_A,\iota_B$ be natural injections of
  $S_{A},S_{B}$ in $S_{[\lceil n/2 \rceil]}$ (preserving the cycle type). For convenience, we use
  $\widehat{h}_A \coloneqq \iota_A(h_A)$ for $h_A \in S_A$ and similarly for $S_B$. Define the coloring $f'
  \colon S_n \to \cT \times \cX$ as
  \begin{align*}
    f'(t_i h_A h_B) & \coloneqq (t_i, f(\widehat{h}_A \widehat{h}_B)),
  \end{align*}
  for every $i \in [k]$, every $h_A \in S_A$ and every $h_B \in S_B$.

  Now we prove that $f'$ is a proper coloring of $\cB_n$. By construction, permutations of different cosets of
  $H$ receive different colors. Let $\sigma$ and $\tau$ be permutations in the same coset $t_i H$ such that
  $\pi \coloneqq \sigma \tau^{-1}$ is a non-trivial cycle, i.e., $\sigma$ and $\tau$ are adjacent in
  $\cB_n$. Write $\sigma = t_i \cdot g_A g_B$ and $\tau = t_i \cdot h_A h_B$ for some $g_A,h_A \in S_A$ and
  $g_B,h_B \in S_B$ so that
  \begin{align*}
    \pi
    & =
    t_i g_A g_B  \cdot h_B ^{-1} h_A^{-1} t_i^{-1}
    =
    t_i (g_A h_A^{-1}) (g_B h_B^{-1}) t_i^{-1},
  \end{align*}
  where the second equality follows because elements of $S_A$ commute with elements of $S_B$. Since $t_i^{-1}
  \pi t_i$ is also a single cycle, exactly one of $(g_A h_A^{-1})$ or $(g_B h_B^{-1})$ must be a single cycle
  and the other the identity. We show that $f'(\sigma) \ne f'(\tau)$ by showing that $f(\widehat{g}_A
  \widehat{g}_B) \ne f(\widehat{h}_A \widehat{h}_B)$. Suppose first that $(g_A h_A^{-1})$ is a cycle and $g_B=
  h_B$. Then
  \begin{align*}
    \widehat{g}_A \widehat{g}_B \cdot (\widehat{h}_A \widehat{h}_B)^{-1}
    & =
    \widehat{g}_A (\widehat{h}_A)^{-1}
  \end{align*}
  is a cycle and thus $f(\widehat{g}_A\widehat{g}_B) \ne f(\widehat{h}_A\widehat{h}_B)$. In the second case we
  have $g_A=h_A$ and $(g_B h_B^{-1})$ is a cycle, so
  \begin{align*}
    \widehat{g}_A \widehat{g}_B \cdot (\widehat{h}_A \widehat{h}_B)^{-1}
    & =
    \widehat{g}_A  (\widehat{g}_B \widehat{h}_B^{-1}) \widehat{g}_A^{-1}
  \end{align*}
  is also a cycle and again $f(\widehat{g}_A\widehat{g}_B) \ne f(\widehat{h}_A\widehat{h}_B)$. Therefore, $f'$
  is a proper coloring of $\cB_n$ with
  \begin{align*}
    \binom{n}{\lceil n/2 \rceil} \cdot \lvert \cX \rvert
  \end{align*}
  colors.
\end{proof}

The idea to improve the coloring construction above by a factor of $2$ is a small generalization of the idea
of \cref{lemma:indep_set_improved} for the independent set.

\begin{lemma}\label{lemma:chi_improved}
  Let $n \ge 4$ be an integer divisible by $4$ and suppose $f \colon S_{n/2} \to \cX$ of $\cB_{n/2}$ is a
  proper coloring that respects signs in the sense that permutations in the same color class have the same
  sign. Then there exists an explicit proper coloring $f'$ of $\cB_n$ that respects signs and with
  \begin{align*}
    \frac{1}{2}\cdot\binom{n}{n/2} \cdot \lvert \cX \rvert
  \end{align*}
  colors.
\end{lemma}

\begin{proof}
  We proceed as in the proof of \cref{lemma:chi_basic}, but instead of assigning a color to each (left) coset
  of $H\coloneqq S_A\times S_B$ (where $A\coloneqq [n/2]$ and $B\coloneqq [n]\setminus[n/2]$) we will be able
  to assign the same color to every two cosets, thereby using only half as many colors. Again, we let
  $\iota_A,\iota_B$ be natural injections of $S_{A},S_{B}$ in $S_{[\lceil n/2 \rceil]}$ and we use the
  notation $\widehat{h}_A \coloneqq \iota_A(h_A)$ for $h_A \in S_A$ and similarly for $S_B$.

  Let $\gamma$ be the product of transpositions $\gamma\coloneqq \prod_{i=1}^{n/2} (i, n/2+i)$ and note that
  since $n$ is divisible by $4$, we have $\sgn(\gamma)=1$.

  Note that since $\gamma\notin H$, it follows that for every $u\in S_n$, we have $uH\neq u\gamma H$. In
  particular, for $k \coloneqq \binom{n}{n/2}$, we can find $u_1,\ldots,u_{k/2}\in S_n$ so that the cosets of
  $H$ are precisely
  \begin{align*}
    u_1 H, u_2 H,\ldots, u_{k/2} H,
    u_1 \gamma H, u_2 \gamma H, \ldots, u_{k/2} \gamma H.
  \end{align*}
  Let $\cU\coloneqq \{u_1,\ldots,u_{k/2}\}$ and define the coloring $f'\colon S_n\to\cU\times\cX$ as
  \begin{align*}
    f'(u_i\cdot h_A h_B) & \coloneqq (u_i, f(\widehat{h}_A\widehat{h}_B));\\
    f'(u_i\gamma\cdot h_A h_B) & \coloneqq (u_i, f(\widehat{h}_A\widehat{h}_B));
  \end{align*}
  for every $i\in[k/2]$, every $h_A\in S_A$ and every $h_B\in S_B$.

  Let us prove that $f'$ is a proper coloring. We classify the edges of $\cB_n$ into the following six types.
  \begin{enumerate}
  \item $\{u_i\cdot g_A g_B, u_j\cdot h_A h_B\}$ for some $i,j\in[k/2]$ with $i\neq j$, some $g_A,h_A\in S_A$
    and some $g_B,h_B\in S_B$.
    \label{it:ui-uj}
  \item $\{u_i \gamma\cdot g_A g_B, u_j \gamma\cdot h_A h_B\}$ for some $i,j\in[k/2]$ with $i\neq j$, some
    $g_A,h_A\in S_A$ and some $g_B,h_B\in S_B$.
    \label{it:uigamma-ujgamma}
  \item $\{u_i \gamma\cdot g_A g_B, u_j\cdot h_A h_B\}$ for some $i,j\in[k/2]$ with $i\neq j$, some
    $g_A,h_A\in S_A$ and some $g_B,h_B\in S_B$.
    \label{it:uigamma-uj}
  \item $\{u_i\cdot g_A g_B, u_i\cdot h_A h_B\}$ for some $i\in[k/2]$, some $g_A,h_A\in S_A$ and some
    $g_B,h_B\in S_B$.
    \label{it:ui-ui}
  \item $\{u_i \gamma\cdot g_A g_B, u_i \gamma\cdot h_A h_B\}$ for some $i\in[k/2]$, some $g_A,h_A\in S_A$ and
    some $g_B,h_B\in S_B$.
    \label{it:uigamma-uigamma}
  \item $\{u_i \gamma\cdot g_A g_B, u_i\cdot h_A h_B\}$ for some $i\in[k/2]$, some $g_A,h_A\in S_A$ and some
    $g_B,h_B\in S_B$.
    \label{it:uigamma-ui}
  \end{enumerate}

  Edges of the types \ref{it:ui-uj}, \ref{it:uigamma-ujgamma} and \ref{it:uigamma-uj} are not monochromatic by
  observing the first coordinate of $f'$. Edges of the types \ref{it:ui-ui} and \ref{it:uigamma-uigamma} are
  not monochromatic by an argument completely analogous to that of \cref{lemma:chi_basic}.

  Let us then consider an edge of type \ref{it:uigamma-ui}. Since
  \begin{align*}
    u_i \gamma\cdot g_A g_B\cdot h_B^{-1} h_A^{-1}\cdot u_i^{-1}
  \end{align*}
  is a cycle, it follows that
  \begin{align*}
    \pi\coloneqq \gamma\cdot g_A g_B\cdot h_B^{-1} h_A^{-1}
  \end{align*}
  is a cycle. But the definition of $\gamma$ implies that $\pi$ does not have any fixed points, so it must be
  a full cycle, hence $\sgn(\pi)=-1$. Therefore, exactly one of $g_A g_B$ and $h_A h_B$ must have negative
  sign, hence $f(\widehat{g}_A\widehat{g}_B)\neq f(\widehat{h}_A\widehat{h}_B)$ as $f$ respects signs, which
  implies that $f'(u_i \gamma\cdot g_A g_B)\neq f'(u_i\cdot h_A h_B)$.

  It remains to show that $f'$ respects signs. It is enough to show that for every $i\in[k/2]$, every
  $g_A,h_A\in S_A$ and every $g_B,h_B\in S_B$ such that $f(\widehat{g}_A\widehat{g}_B) =
  f(\widehat{h}_A\widehat{h}_B)$, the following permutations have the same sign
  \begin{align*}
    u_i\cdot g_A g_B, & &
    u_i \gamma\cdot g_A g_B, & &
    u_i\cdot h_A h_B, & &
    u_i \gamma\cdot h_A h_B.
  \end{align*}
  Since $\sgn(\gamma)=1$, the first two permutations have the same sign. The same argument shows the last two
  have the same sign. Hence, it is enough to show that $\sgn(u_i\cdot g_A g_B) = \sgn(u_i\cdot h_A h_B)$,
  which is equivalent to $\sgn(g_A g_B) = \sgn(h_A h_B)$ and this follows from the fact that
  $f(\widehat{g}_A\widehat{g}_B) = f(\widehat{h}_A\widehat{h}_B)$ and that $f$ respects signs.
\end{proof}

Similarly to the independence number, when $n$ is a power of $2$, the factor of $2$ advantage of
\cref{lemma:chi_improved} compounds to a total advantage of $n/2$ in the final construction of a proper
coloring.

\TheoMainChromNumb*

\begin{proof}
  The first part of the theorem follows by a induction in $n$ using \cref{lemma:chi_basic} (the base case of
  $n=1$ consists of a trivial coloring). The second part follows by induction in $\log_2(n)$ using
  \cref{lemma:chi_improved} instead and base cases of $n=1$ and $n=2$, in which the coloring is rainbow.
\end{proof}

\subsection*{Acknowledgements}

We thank Madhur Tulsiani for bringing the Kane--Lovett--Rao result~\cite{lovett17} to our attention.

\bibliographystyle{alpha}
\bibliography{references}

\appendix

\section{Tail Bounds Asymptotics}\label{app:tail_bounds}

The tail values $T_n^{\ell,k_0}(c)$ from the linear program II, in \cref{def:lp_2}, can be bounded by
$T^{\ell,k_0}(c)$ from the linear program III in \cref{def:lp_3} as long as $n$ is sufficiently large. More
precisely, our goal in this section is to prove the following.

\begin{restatable}{claim}{TailBoundingTail}\label{claim:tail_bounding_tail}
  Let $k_0$ be an odd positive integer and $\ell$ be a non-negative even integer. Then
  \begin{align*}
    \lim_{\substack{n \to \infty\\\text{odd}}} T_n^{\ell,k_0}(c) & \le T^{\ell,k_0}(c).
  \end{align*}
\end{restatable}

First, we recall reasonably sharp bounds on the Stirling's approximation in \cref{subsec:stirling_approx},
then we derive a simple result about a series closely related to the geometric series in
\cref{subsec:series_related_to_geom} and we finally relate $T_n^{\ell,k_0}(c)$ and $T^{\ell,k_0}(c)$ in
\cref{subsec:comparing_tail_bounds}.

\subsection{Stirling Approximation}\label{subsec:stirling_approx}

We rely on Robbins' version of Stirling's approximation.

\begin{theorem}[Robbins' version of Stirling's approximation~\cite{Robbins55}]
  \label{thm:robbins}
  For $n\in\NN_+$, we have
  \begin{align*}
    n! & = \sqrt{2\pi n} \cdot \left(\frac{n}{e}\right)^n \cdot e^{F(n)},
  \end{align*}
  where
  \begin{align*}
    \frac{1}{12n+1} & < F(n) < \frac{1}{12n}.
  \end{align*}
\end{theorem}

The approximation above gives us the following approximation of the \emph{``middle binomial''}.

\begin{corollary}\label{cor:stirling}
  For $n\in\NN_+$, we have
  \begin{align*}
    \binom{2n}{n} & = \frac{2^{2n}}{\sqrt{n\pi}} e^{F(2n)-2F(n)},
  \end{align*}
  where $F(n)$ is as in \cref{thm:robbins}. In particular, we have
  \begin{align*}
    \binom{2n}{n} \ge \frac{2^{2n}}{\sqrt{n\pi}} e^{-\frac{2}{15n}},
  \end{align*}
\end{corollary}

\subsection{Series Related to the Geometric Series}\label{subsec:series_related_to_geom}

We will need a closed form expression for a series related to the geometric series.

\begin{claim}\label{claim:series_closed_form}
  For $\lvert q\rvert < 1$, $a\in\NN_+$
  and $n_0,b\in\NN$, we have
  \begin{align*}
    \sum_{n\geq n_0} (an+b) q^{an + b}
    & =
    \frac{q^b}{(1 - q^a)^2} ((a n_0 + b) q^{a n_0} + (a - a n_0 - b) q^{a n_0 + a}).
  \end{align*}
\end{claim}

\begin{proof}
  Indeed
  \begin{align*}
    \sum_{n \ge n_0} (an+b) q^{an + b}
    & =
    q \sum_{n \ge n_0} (an+b) q^{an + b -1}
    \\
    & =
    q \left(\frac{d}{dt} \sum_{n \ge n_0} t^{an + b} \middle)\right\vert_{t=q}
    \\
    & =
    q \left(\frac{d}{dt} \frac{t^{an_0 + b}}{1-t^a} \middle)\right\vert_{t=q}
    \\
    & =
    q \frac{(an_0 +b) q^{an_0 + b -1}(1-q^a) + a q^{an_0 + b + a -1}}{(1-q^a)^2},
  \end{align*}
  which simplifies to the claimed bound.
\end{proof}

\subsection{Comparing Tail Bounds}\label{subsec:comparing_tail_bounds}

We proceed to relate $T_n^{\ell,k_0}(c)$ and $T^{\ell,k_0}(c)$, but first we will need two technical claims.

\begin{claim}\label{claim:series_to_limit_tail}
  Let $k_0$ be a positive odd integer. We have
  \begin{align*}
    \sum_{\substack{k=k_0 + 2 \\ k \text{odd}}}^{\infty} \frac{c^{2k+2\ell}}{\binom{2k+2\ell}{k+\ell}}
    & \le
    T^{\ell,k_0}(c).
  \end{align*}
\end{claim}

\begin{proof}
  Indeed, we have
  \begin{align*}
    \sum_{\substack{k = k_0 + 2 \\ k\text{ odd}}}^{\infty} \frac{c^{2k+2\ell}}{\binom{2k+2\ell}{k+\ell}}
    & \le
    \sum_{\substack{k=k_0 + 2 \\ k\text{ odd}}}^{\infty}
    \frac{c^{2(k+\ell)}}{\frac{2^{2(k+\ell)}}{\sqrt{\pi(k+\ell)}} e^{-\frac{2}{15(k+\ell)}}}
    & &
    \text{(By \cref{cor:stirling}.)}
    \\
    & \le
    \sqrt{\pi} \cdot e^{\frac{2}{15(k_0+\ell+2)}}
    \sum_{\substack{k=k_0 + 2 \\ k\text{ odd}}}^{\infty} \sqrt{k+\ell} \left(\frac{c}{2}\right)^{2(k+\ell)}
    \\
    & \le
    \sqrt{\pi} \cdot e^{\frac{2}{15(k_0+\ell+2)}}
    \sum_{\substack{k=k_0 + 2 \\ k\text{ odd}}}^{\infty} (k+\ell) \left(\frac{c}{2}\right)^{2(k+\ell)}
    \\
    & =
    \frac{\sqrt{\pi}}{2} \cdot e^{\frac{2}{15(k_0+\ell+2)}}
    \sum_{i=(k_0 + 1)/2}^{\infty} (4i + 2\ell + 2)\left(\frac{c}{2}\right)^{4i+ 2 \ell + 2}
    & &
    \text{(Change $k$ to $2i+1$.)}
    \\
    & =
    \frac{\sqrt{\pi}}{2} \cdot e^{\frac{2}{15(k_0+\ell+2)}} \cdot
    \frac{%
      \left(2\left(k_0 + \ell \right) + 4 \right) \left(\frac{c}{2}\right)^{2(k_0 + \ell)+4}
      - 2(k_0 + \ell)\left(\frac{c}{2}\right)^{2(k_0+\ell)+8}
    }{%
      \left(1- \left(\frac{c}{2}\right)^4\right)^2
    }
    & &
    \text{(By \cref{claim:series_closed_form}.)}
    \\
    & \le
    T^{\ell,k_0}(c).
  \end{align*}
\end{proof}

We will need a bound on the ratio between a middle binomial and a defective middle binomial as follows.

\begin{claim}\label{claim:approx_binom_bound}
  If $d$ and $s \ge 2d$ are non-negative integers, then
  \begin{align*}
    \frac{\binom{2s}{s}}{\binom{2s-d}{s}} & \le 3^d.
  \end{align*}
\end{claim}

\begin{proof}
  Since $s \ge 2d$, we get
  \begin{align*}
    \frac{\binom{2s}{s}}{\binom{2s-d}{s}}
    & =
    \frac{(2s)!(s-d)!}{s!(2s-d)!}
    =
    \frac{(2s)_d}{(s)_d}
    =
    \prod_{i=0}^{d-1} \frac{2s-i}{s-i}
    =
    \prod_{i=0}^{d-1} \left(2 + \frac{i}{s-i} \right) \le 3^d.\qedhere
  \end{align*}
\end{proof}

Now we are ready to prove the main result of this section, which we restate below for convenience.

\TailBoundingTail*

\begin{proof}
  For a positive odd integer $n$, let
  \begin{align*}
    T
    & \coloneqq
    \sum_{\substack{k = k_0 + 2 \\ k\text{ odd}}}^{\infty}
    \frac{c^{2k+2\ell}}{\binom{2k+2\ell}{k+\ell}}
    &
    T'_n
    & \coloneqq
    \sum_{\substack{k=\frac{n-3\ell-1}{2} + 1 \\ k\text{ odd}}}^{\frac{n-\ell-1}{2}}
    \frac{c^{n-1}}{\binom{n-\ell-1}{k+\ell}}.
  \end{align*}

  In the sum of $T_n'$, note that $(n-3\ell-1)/2 \le k \le (n-\ell-1)/2$ implies $\lvert (n - \ell - 1) -
  2(k+\ell) \rvert \le 2 \ell$ and $\lvert (n-1) -2(k+\ell)\rvert \le \ell$. By
  \cref{claim:approx_binom_bound} for $n \ge 9\ell + 1$, we have
  \begin{align*}
    T_n'
    & \le
    c^{\ell} \cdot 3^{2\ell} \sum_{\substack{k=\frac{n-3\ell-1}{2} + 1 \\ k\text{ odd}}}^{\frac{n-\ell-1}{2}}
    \frac{c^{2(k+\ell)}}{\binom{2(k+\ell)}{k+\ell}}.
  \end{align*}
  Observe that if $n \ge 2 k_0 + 3\ell +3$, the sum above is contained in the tail of $T$, which is a
  convergent series by \cref{claim:series_to_limit_tail}. Hence, $T_n' \xrightarrow[\text{odd}]{n \to \infty}
  0$. Therefore, we obtain
  \begin{align*}
    \lim_{\substack{n \to \infty\\\text{odd}}} T_n^{\ell,k_0}(c)
    & =
    \lim_{\substack{n \to \infty\\\text{odd}}} \sum_{\substack{k=k_0 + 2 \\ k\text{ odd}}}^{\frac{n-3\ell-1}{2}}
    \frac{c^{2k+2\ell}}{\binom{2k+2\ell}{k+\ell}}
    + \lim_{\substack{n \to \infty\\\text{odd}}}
    \sum_{\substack{k=\frac{n-3\ell-1}{2} + 1 \\ k\text{ odd}}}^{\frac{n-\ell-1}{2}}
    \frac{c^{n-1}}{\binom{n-\ell-1}{k+\ell}}\\
    & =  T  + \lim_{\substack{n \to \infty\\\text{odd}}} T'_n
    =  T  \le  T^{\ell,k_0}(c) ,
  \end{align*}
  where the last inequality follows from \cref{claim:series_to_limit_tail} again.
\end{proof}

\section{KLR Proofs}\label{app:klr_proofs}

For the reader's convenience we recall some proofs either from~\cite{lovett17} or implicit in it.

\KlrNonNegChar*

\begin{proof}
  We have
  \begin{align*}
    \chi^{\lambda}(\phi_{A})
    & =
    \frac{1}{\lvert A \rvert^2} \sum_{\pi,\pi' \in A} \chi^{\lambda}(\pi(\pi')^{-1})
    \\
    & =
    \frac{1}{\lvert A \rvert^2} \sum_{\pi,\pi' \in A}\tr(S^{\lambda}(\pi)S^{\lambda}(\pi')^{\top})
    =
    \tr(S^{\lambda}(\xi)S^{\lambda}( \xi)^{\top})
    \ge
    0,
  \end{align*}
  where $\xi \coloneqq \sum_{\pi \in A} \pi/\lvert A \rvert$.
\end{proof}

Using the pseudorandomness condition, we can bound the character of $M^{\mu}$ on $\phi_{A}$. In order to do
so, observe that the action of $S_n$ on $[n]_k$ corresponds precisely to the action of $S_n$ on the Young
module $M^{h^n_k}$ corresponding to the hook $h^n_k$ of leg $k$ (since the leg of a tabloid $[T]$ of shape
$h^n_k$ corresponds to a tuple in $[n]_k$ and the order of elements in the first row of $[T]$ is
arbitrary). In this case, we refer to tuples and tabloids interchangeably.

\KlrYoundModuleBound*

\begin{proof}
  Let $\mu \coloneqq h^n_k$. We explore the uniformity of the action of $S_n$ on $[n]_k$ given by the
  $(k,r)$-pseudorandomness assumption. Consider the matrix representation of $M^{\mu}$ indexed by
  $k$-tuples. More precisely, for $\pi \in S_n$ and $I,J \in [n]_k$, we have
  \begin{align*}
    M^{\mu}_{I,J}(\pi) & \coloneqq 1_{[\pi(J)=I]}.
  \end{align*}
  Set $\xi \coloneqq \sum_{\pi \in A} \pi/\lvert A \rvert$. Then $(k,r)$-pseudorandomness yields
  \begin{align*}
    M^{\mu}_{I,J}(\xi)
    & =
    \Pr_{\pi \in A}[\pi(J) = I]
    <
    \frac{r}{(n)_k}.
  \end{align*}
  Hence
  \begin{align*}
    \tr(M^{\mu}(\phi_{A}))
    & =
    \tr(M^{\mu}(\xi)M^{\mu}(\xi)^{\top})
    \\
    & =
    \sum_{I,J \in [n]_k}  M^{\mu}(\xi)_{I,J} \cdot M^{\mu}(\xi)_{I,J}
    \\
    & \le
    \sum_{I,J \in [n]_k} \frac{r}{(n)_k} \cdot M^{\mu}(\xi)_{I,J}
    =
    r,
  \end{align*}
  where the last equality follows from $M^{\mu}$ having exactly one entry of value $1$ and all others zero in
  each column.
\end{proof}

We can bound an arbitrary non-trivial character in terms of the pseudorandomness parameter and an appropriate
Kostka number.

\KlrIrrepCharBound*

\begin{proof}
    By the Young's rule, \cref{theo:young_decomp}, we obtain
    \begin{align*}
      \tr(M^{h^n_k}(\phi_A)) = \sum_{\lambda' \vdash n} K_{\lambda',h^n_k} \cdot \chi^{\lambda'}(\phi_A).
    \end{align*}
    Using the bound on $\tr(M^{h^n_k}(\phi_A))$ from \cref{claim:young_module_wp_bound} gives
    \begin{align*}
      \sum_{\lambda' \vdash n} K_{\lambda',h^n_k} \cdot \chi^{\lambda'}(\phi_A)
      & \le
      r.
    \end{align*}
    Since $\chi^{(1^n)}(\phi_A) = K_{(1^n),h^n_k} = 1$ and $\chi^{\lambda'}(\phi_A) \ge 0$ for every $\lambda'\vdash n$ from
    \cref{fact:non_negativity_of_irrep}, we have
    \begin{align*}
      1 + K_{\lambda,h^n_k} \cdot \chi^{\lambda}(\phi_A)
      & \le
      r,
    \end{align*}
    and the bound follows.
\end{proof}

\end{document}